\documentclass[11pt, a4paper, reqno]{amsart}

\usepackage[usenames,dvipsnames]{xcolor}
\usepackage{
  amssymb
  ,amsfonts
  ,url
  ,amsthm
  ,amsmath
  ,etex
  ,cancel
  ,epigraph
  ,enumitem
  ,mathtools
  ,stmaryrd
  ,marginnote
  ,todonotes
  ,manfnt
  ,lmodern
  ,etoolbox
  ,microtype
  ,tikz
  ,pifont
  ,hyphenat
  ,datetime
  ,hyperref
  ,mxedruli
  ,bold-extra
  ,CJKutf8
  ,tipx
  ,csquotes
  ,cjhebrew
  ,enumitem}

\usepackage{xspace}
	\providecommand{\abbrv}[1]{#1.\@\xspace}
	\providecommand{\ie}{\abbrv{i.e}}

  \newcommand{\adef}{\abbrv{Def}}
  
  \newcommand{\aprop}{\abbrv{Prop}}

\def\rec{\text{\mxedb r}}
\newcommand{\georgian}[1]{\text{\mxedb #1}}
\SetSymbolFont{stmry}{bold}{U}{stmry}{m}{n}
\usepackage[cmsy]{MnSymbol}
\usepackage[german,   
      spanish,  
      greek,    
      english]{babel}

\usepackage[utf8]{inputenc}
\usepackage[T1]{fontenc}

\usepackage{mathrsfs}

\newcommand{\omissis}{[\dots\unkern]}

\def\ho{\text{Ho}}

\usepackage[titletoc]{appendix}
\usepackage[bottom=6cm,
      left=4cm,
      right=4.5cm,
      top=4.5cm]{geometry}
\usepackage[all,2cell, cmtip]{xy}\UseAllTwocells

\hypersetup{%
  unicode=true,          
  pdftoolbar=true,        
  pdfmenubar=true,        
  pdffitwindow=true,     
  pdfstartview={FitH},    
  pdftitle={recollements},
    pdfauthor={D. Fiorenza and F. Loregian},
  colorlinks=true,
  linkcolor=black,
  urlcolor=blue!70,
  citecolor=blue!70}

\newcommand{\xto}[1]{\xrightarrow{#1}}
\newcommand{\xot}[1]{\xleftarrow{#1}}

\renewcommand{\phi}{\varphi}

\newcommand{\tee}{\mathfrak{t}}

\providecommand{\tenuis}{\scalebox{.7}[.6]{\ooalign{\hss\textbar\hss\cr\hss$\equiv$\hss}}}
\providecommand{\glue}{\mathbin{\scalebox{.8}[.6]{\ooalign{\hss{\raisebox{4.45pt}{$\cup$}}\hss\cr\hss\tenuis\hss}}}}
\providecommand{\smalltenuis}{\scalebox{.7}[.6]{\ooalign{\hss\textbar\hss\cr\hss$\equiv$\hss}}}
\providecommand{\smallglue}{\mathbin{\scalebox{.8}[.6]{\ooalign{\hss{\raisebox{4pt}{$\cup$}}\hss\cr\hss\smalltenuis\hss}}}}

\providecommand{\var}[2]{ \left[ \begin{smallmatrix} %
      #1 \\ \downarrow \\ #2 %
      \end{smallmatrix} \right]}

\newcommand{\refbf}[1]{\textbf{\ref{#1}}}

\makeatletter
  \def\@cite#1#2{[\textbf{#1}\if@tempswa , #2\fi]}
  \def\@biblabel#1{[\textsf{#1}]}
\makeatother

\DeclareMathOperator{\id}{id}       
\DeclareMathOperator{\fib}{fib}

\newcommand{\cate}[1]{\text{\fontseries{b}\selectfont{\upshape #1}}}
\renewcommand{\textbf}[1]{\text{\fontseries{b}\selectfont{\upshape #1}}}
\newcommand{\D}{\cate{D}}
\newcommand{\fF}{\mathbb{F}}
\newcommand{\mcal}[1]{\mathcal{#1}}

\newcommand{\CC}{\cate{C}}
\newcommand{\EE}{\mathcal{E}}
\newcommand{\MM}{\mathcal{M}}
\newcommand{\K}{\cate{K}}
\newcommand{\ts}{\textsc{ts}}

\long\def\symbolfootnote[#1]#2{\begingroup%
\def\thefootnote{\fnsymbol{footnote}}\footnote[#1]{#2}\endgroup}

\newcommand{\lrlarrows}{\mathbin{\substack{\leftarrow\\[-.9em] \rightarrow \\[-.9em] \leftarrow}}}
\def\dgrm#1{\mathsf{#1}}
\newcommand{\iddots}{\protect{\rotatebox[origin=c]{90}{$\ddots$}}}

\renewcommand{\setminus}{\smallsetminus}

\def\uno{\text{\ding{172}} }
\def\due{\text{\ding{173}} }

\renewcommand*{\thefootnote}{(\textbf{\arabic{footnote}})}

\newtheoremstyle{reference}%
   {}                %
   {}                %
   {}              
   {}                      
   {\bfseries\fontseries{b}\upshape}              
   {:}                     
   {.2em}                  
   {\thmname{#1}           
    \thmnumber{#2}         
    \thmnote{\bfseries\fontseries{b}\upshape [#3]}}  

\theoremstyle{reference}
  \newtheorem{theorem}{Theorem}[section]
  \newtheorem{lemma}[theorem]{Lemma}
  \newtheorem{proposition}[theorem]{Proposition}
  \newtheorem{example}[theorem]{Example}
  
  \newtheorem{remark}[theorem]{Remark}
  \newtheorem{definition}[theorem]{Definition}
  
  \newtheorem{notat}[theorem]{Notation}
  
  \newtheorem{scholium}[theorem]{Scholium}
  \newtheorem*{theorem*}{Theorem}
  \newtheorem*{lemma*}{Lemma}
  \newtheorem*{proposition*}{Proposition}
  \newtheorem*{example*}{Example}
  \newtheorem*{exercise*}{Exercise}
  \newtheorem*{remark*}{Remark}
  \newtheorem*{definition*}{Definition}
  \newtheorem*{corollary*}{Corollary}
  \newtheorem*{notat*}{Notation}
  \newtheorem*{scholium*}{Scholium}

\providecommand{\glue}{\mathbin{\mathpalette\dopawglue\relax}}

\providecommand{\cupdot}{\mathbin{\mathpalette\docupdot\relax}}
\newcommand{\docupdot}[2]{%
  \ooalign{$#1\cup$\cr\hfil$#1\cdot$\hfil}}

\definecolor{semilightgray}{rgb}{0.65, 0.65, 0.65}

\def\gray#1{\textcolor{semilightgray}{#1}}

\def\DefaultEpigraphWidth{.6\textwidth}

\begin{document}
\shorthandoff{"}

\title{Recollements in stable $\infty$-categories}

\author{Domenico Fiorenza, Fosco Loregi\`an}
\address{$\bullet$ \textsf{Domenico Fiorenza}: Dipartimento di Matematica ``Guido Castelnuovo'', 
Universit\`a degli Studi di Roma ``la Sapienza'',	
P.le Aldo Moro 2 -- \oldstylenums{00185} -- Roma.} 
\email{fiorenza@mat.uniroma1.it}

\address{$\bullet$ \textsf{Fosco Loregian}: SISSA - 
Scuola Internazionale Superiore di Studi Avanzati, 
via Bonomea 265, \oldstylenums{34136} Trieste.} 
\email{floregi@sissa.it}
\begin{abstract}
We develop the theory of \emph{recollements} in a stable $\infty$-ca\-te\-go\-ri\-cal setting. In the axiomatization of Be\u\i linson, Bernstein and Deligne, recollement situations provide a generalization of Grothen\-dieck's ``six functors'' between derived categories. The adjointness relations between functors in a recollement $\D^0\lrlarrows \D \lrlarrows \D^1$ induce a ``recoll\'ee'' $t$-structure $\tee_0\smallglue\tee_1$ on $\D$ , given $t$-structures $\tee_0,\tee_1$ on $\D^0, \D^1$. Such a classical result, well-known in the setting of triangulated categories, {\color{black}is recasted in the setting of stable $\infty$-categories and the properties of the associated ($\infty$-categorical) factorization systems are investigated.}
%
In the geometric case of a stratified space, various recollements arise, which ``interact well'' with the combinatorics of the intersections of strata to give a well-defined, associative $\smallglue$ operation. From this we deduce a generalized associative property for $n$-fold gluing $\tee_0\smallglue\cdots\smallglue \tee_n$, valid in any stable $\infty$-category.
\end{abstract}
\maketitle
\tableofcontents
\section{Introduction.}\index{Perverse sheaf}
Recollements in triangulated categories were introduced by A\@. Be\u{\i}linson,  J\@. Bernstein and P\@. Deligne in \cite{BBDPervers}, searching an axiomatization of the Grothendieck's ``six functors'' formalism for derived categories of sheaves on (the strata of a) stratified topological space. \cite{BBDPervers} will be our main source of inspiration, and reference for classical results and computations; among other recent but standard references, we mention \cite{KS1,Banagl}. Later, ``recollement data'' were noticed to appear quite naturally in the context of intersection homology \cite{pflaum2001analytic,goresky1980intersection,goresky1983intersection} and Representation Theory \cite{parshall1988derived,kiehl2001weil}. In more recent years Beligiannis and Reiten \cite{Beligiannisreiten}, adapting to the triangulated setting an old idea of Jans \cite{jans1965}, linked recollement data to so-called \textsc{ttf}-triples \index{ttf triple@\smallcap{ttf}-triple}(\ie triples $(\mcal{X}, \mcal{Y}, \mcal{Z})$ such that both $(\mcal{X}, \mcal{Y})$ and $(\mcal{Y}, \mcal{Z})$ are $t$\hyp{}structures): recollement data, in the form of \textsc{ttf}-triples, appear quite naturally studying derived categories of representations of algebras, see \cite[\abbrv{Ch} \textbf{4}]{Beligiannisreiten}.

{\color{black}Here we
translate the basic theory of recollements in the stable $\infty$-categorical setting and investigate their properties. 
In particular, inspired by the analysis of geometric recollements data associated with a stratified space, we consider the problem of associativity for iterated recollements, and show how 
one has associativity as soon as the relevant Beck-Chevalley condition is satisfied. Remarkably, in the geometric situation, this condition is always satisfied so that, as one should maybe  expect, geometric iterated recollements do not depend on the order on which recollement data are used to produce the global $t$-structure on the derived category of the stratified space $X$. Although probably implicit in the construction, this remark appears not be spelled out explicitly in \cite{BBDPervers}.}
%

\section{Classical Recollements.}\label{classical}\index{Recollement}
\newfont{\scaledfont}{yswab scaled 1000}
\setlength{\epigraphwidth}{.5\textwidth}
\epigraph{
	{\scaledfont Sitzt ihr nur immer! leimt zusammen},\\
	{\scaledfont Braut ein Ragout von andrer Schmaus},\\
	{\scaledfont und blas't die k\"ummerlichen Flammen}\\ 
	{\scaledfont aus eurem Aschenh\"aufchen `raus}!}{\textsc{Faust}, I \oldstylenums{538}-\oldstylenums{541}.}
\setlength{\epigraphwidth}{\DefaultEpigraphWidth}
The aim of this subsection is to present the basic features of ``classical'' recollements in the setting of stable $\infty$\hyp{}categories ignoring, for the moment, the translation in terms of normal torsion theories which will follow.
\begin{definition}\label{def:recol}
A (\emph{donnée de}) \emph{recollement} consists of the following arrangement of stable $\infty$\hyp{}categories and functors between them:
\begin{equation}\label{rec}
\xymatrix{
  \D^0	& \D	& \D^1
  \ar|i "1,1";"1,2" 
  \ar@<8pt>^{i_L} "1,2";"1,1" 
  \ar@<-8pt>_{i_R} "1,2";"1,1" 
  \ar|q "1,2";"1,3" 
  \ar@<8pt>^{q_L} "1,3";"1,2" 
  \ar@<-8pt>_{q_R} "1,3";"1,2" 
}
\end{equation}
satisfying the following axioms:
\begin{enumerate}
\item There are adjunctions $i_L\dashv i\dashv i_R$ and $q_L\dashv q\dashv q_R$;
\item The counit $\epsilon_{(i_L\dashv i)} \colon i_L i \to 1$ and the unit $\eta_{(i\dashv i_R)}\colon 1\to i_r i$ are natural isomorphisms; also, the unit $1\to qq_R$ and counit $qq_L \to 1$ are natural isomorphisms;\footnote{With a little abuse of notation we will write $i_L i = \id_{\D^0} = i_R i$, and similarly for $qq_L = \id_{\D} = qq_R$.}
\item The (essential) image of $i$ equals the \emph{essential kernel} of $q$, namely the full subcategory of $\D$ such that $qX\cong 0$ in $\D^1$;
\item 
The natural homotopy commutative diagrams 
\begin{equation}
\xymatrix{
  q_L q & \id_\D  & ii_R  & \id_\D \\
  0 & ii_L  & 0 & q_R q
  \ar "1,1";"1,2" ^{\epsilon_{(q_L\dashv q)}}
  \ar "1,1";"2,1" 
  \ar "1,2";"2,2" ^{\eta_{(i_L\dashv i)}}
  \ar "1,3";"1,4" ^{\epsilon_{(i\dashv i_R)}}
  \ar "1,3";"2,3" 
  \ar "1,4";"2,4" ^{\eta_{(q\dashv q_R)}}
  \ar "2,1";"2,2" 
  \ar "2,3";"2,4" 
}
\end{equation}
induced by axioms~(\oldstylenums{1}), (\oldstylenums{2}) and (\oldstylenums{3})
are pullouts\footnote{Here and everywhere else the category of functors to a stable $\infty$\hyp{}category becomes a stable $\infty$\hyp{}category in the obvious way (see \cite[\abbrv{Prop} \textbf{1.1.3.1}]{LurieHA}).}.
\end{enumerate}
\end{definition}
\begin{remark}
As an immediate consequence of the axioms, a recollement gives rise to various reflections and coreflections of $\D$: since by axiom (\oldstylenums{2}) the functors $i, q_L, q_R$ are all fully faithful, $q_Rq, ii_L$ are reflections and $q_Lq, ii_R$ are coreflections. Moreover, axioms (\oldstylenums{3}) and (\oldstylenums{4}) entail that the compositions $i_R q_R, qi, i_L q_L$ are all ``exactly'' zero, \ie not only the kernel of $q$ is the essential image of $i$, but also the kernel of $i_{L/R}$ is the essential image of $q_{L/R}$.
\end{remark}
\begin{remark}
Axioms (\oldstylenums{2}) and (\oldstylenums{4}) together imply that there exists a canonical natural transformation $i_R \to i_L$, obtained as $i_R(\eta_{(i_L \dashv i)})$ (or equivalently, as $i_L(\epsilon_{(i\dashv i_R)})$: it's easy to see that these two arrows coincide). Axiom (\oldstylenums{4}) entails that there is a fiber sequence of natural transformations
\begin{equation}
\xymatrix{
   i_Rq_Lq \ar[r]\ar[d]& i_R \ar[r]\ar[d] &  0\ar[d] \\
   0 \ar[r] &  i_L \ar[r] &  i_Lq_Rq 
}
\notag\end{equation}
\end{remark}
\begin{notat}
We will generally use a compact form like 
\begin{equation}
(i,q)\colon \D^0 \lrlarrows  \D \lrlarrows  \D^1
\end{equation}
to denote a recollement (\refbf{rec}), especially in inline formulas. Variations on this are possible, either to avoid ambiguities or to avoid becoming stodgy. 

We will for example say that ``$(i,q)$ is a recollement on $\D$'' or that ``$\D$ is the \emph{d\'ecollement} of $\D^0, \D^1$'' to denote that there exists a diagram like $(\refbf{rec})$ having $\D$ as a central object. In other situations we adopt an extremely compact notation, referring to a (donn\'e de) recollement with the symbol $\rec$ of (the letter \emph{rae} of the Georgian alphabet, in the {\georgian{mxedruli}} script, see \cite{hewitt1995georgian}). 
\end{notat}
\paragraph{A geometric example.}\index{Stratified space}\index{Recollement!geometric ---} The most natural example of a recollement comes from the theory of \emph{stratified spaces} \cite{Wein,Banagl}:
\begin{example}\label{geomrecoll}
Let $X$ be a topological space, $F \subseteq X$ a closed subspace, and $U=X\setminus F$ its open complement. 

From the two inclusions $j\colon F \hookrightarrow X$, and $i\colon U\hookrightarrow X$ we obtain the adjunctions $j^*\dashv j_*\dashv j^!$, $i_!\dashv i^*\dashv i_*$ between the categories $\mathbf{Coh}(U), \mathbf{Coh}(X)$ and $\mathbf{Coh}(F)$ of coherent sheaves on the strata. Passing to their (bounded below-)derived versions we obtain functors\footnote{For a topological space $A$ we denote $\D(A)$ the \emph{derived $\infty$\hyp{}category} of coherent sheaves on $A$ defined in \cite[\S\textbf{1.3.2}]{LurieHA}; we also invariably denote as $j^*\dashv j_*\dashv j^!$, $i_!\dashv i^*\dashv i_*$ the functors between stable $\infty$\hyp{}categories induced by the homonym functors between abelian categories.}
\begin{equation}
\xymatrix{
  \D(F)	& \D(X)	& \D(U)
  \ar "1,1";"1,2" ^{j_*}
  \ar "1,2";"1,3" ^{i^*}
}
\end{equation}
giving rise to reflections and coreflections
\begin{equation}
\xymatrix{
  \D(F)	& \D(X)	& \D(U)	& \D(F)	& \D(X)	& \D(U).
  \ar@{}|\top "1,1";"1,2" 
  \ar@<4pt>@{^{(}->} "1,1";"1,2" ^{j_*}
  \ar@<4pt> "1,2";"1,1" ^{j^*}
  \ar@{}|\top "1,2";"1,3" 
  \ar@<4pt> "1,2";"1,3" ^{i^*}
  \ar@<4pt>@{^{(}->} "1,3";"1,2" ^{i_!}
  \ar@{}|\perp "1,4";"1,5" 
  \ar@<4pt>@{^{(}->} "1,4";"1,5" ^{j_*}
  \ar@<4pt> "1,5";"1,4" ^{j^!}
  \ar@<4pt> "1,5";"1,6" ^{i^*}
  \ar@{}|\perp "1,5";"1,6" 
  \ar@<4pt>@{^{(}->} "1,6";"1,5" ^{i_*}
}
\end{equation}
These functors are easily seen to satisfy axioms (\oldstylenums{1})-(\oldstylenums{4}) above: see \cite[\textbf{1.4.3.1-5}]{BBDPervers} and \cite[\textbf{7.2.1}]{Banagl} for details.
\end{example}
\begin{remark}
The above example, first discussed in \cite{BBDPervers}, is in some sense para\-digmatic, and it can be seen as a motivation for the abstract definition of recollement: a generalization of Grothendieck's ``six functors'' formalism. Several sources \cite{han2014recollements,bazzoni2013recollements,hugel2011recollements,chen2014recollements} convey the intuition that a recollement $\rec$ is some sort of ``exact sequence'' of triangulated categories, thinking $\D$ as decomposed into two parts, an ``open'' and a ``closed'' one. This also motivates the intuition that a donn\'ee de recollement is not symmetric.
\end{remark}
\paragraph{An algebraic example.}
The algebraic counterpart of the above example involves derived categories of algebras: we borrow the following discussion from \cite{han2014recollements}.
\begin{example}
Let $A$ be an algebra, and $e\in A$ be an idempotent element; let $J = eAe$ be the ideal generated by $e$, and suppose that
\begin{itemize}
\item $Ae\otimes_J eA\cong J$ under the map $(xe, ey)\mapsto xey$;
\item $\text{Tor}_n^J(Ae, eA)\cong 0$ for every $n > 0$.
\end{itemize}
Then there exists a recollement
\begin{equation}
\xymatrix@C=2cm{
	\D(A/J) \ar[r]^{i=- \otimes_{A/J} A/J} & \D(A)\ar[r]^{q = -\otimes_A Ae} & \D(eAe)
}
\end{equation}
\end{example}
between the derived categories of modules on the rings $A/J, A, eAe$.

Interestingly enough, also this example is paradigmatic in some sense; more precisely, \emph{every} recollement $\rec\colon \D(A_1)\lrlarrows \D(A) \lrlarrows \D(A_2)$ is equivalent, in a suitable sense, to a ``standard'' recollement where $i_L$ and $q_L$ act by tensoring with distinguished objects $Y\in \D(A), Y_2\in \D(A_2)$.
\begin{definition}[Standard recollement]
Let {\georgian{s}} $\colon \D(A_1)\lrlarrows \D(A) \lrlarrows \D(A_2)$ be a recollement between algebras; it is called a \emph{standard} recollement generated by a pair $(Y,Y_2)$ if $i_L \cong -\otimes_A Y$, and $q_L \cong -\otimes_{A_2}Y_2$.
\end{definition}
\begin{proposition}
Let $\rec \colon \D(A_1)\lrlarrows \D(A) \lrlarrows \D(A_2)$ be a recollement between algebras; then $\rec$ is equivalent (in the sense of Remark \refbf{equiv.of.reco}) to a standard recollement $\text{\georgian{s}}$ generated by the pair $(Y,Y_2)$.
\end{proposition}
The proof relies on the following
\begin{lemma}
Let $A_1, A, A_2$ be algebras. The derived categories on these algebras are part of a recollement $\colon \D(A_1)\lrlarrows \D(A) \lrlarrows \D(A_2)$ if and only if there exist two objects $X_1, X_2 \in\D(A)$ such that
\begin{itemize}
\item $\hom(X_i, X_i)\cong A_i$ for $i=1,2$;
\item $X_2$ is an exceptional and compact object, and $X_1$ is exceptional and self-compact;
\item $X_1 \in \{ X_2\}^\perp$;
\item $\{X_1\}^\perp\cap\{X_2\}^\perp = (0)$.
\end{itemize}
\end{lemma}
See \cite[\S\textbf{2}]{han2014recollements} for details.
\paragraph{A homotopical example.}
Let $\ho({}_\mcal{G}\cate{Sp})$ be the \emph{global stable homotopy category} of \cite{SchwedeGlobal}; this is defined as the localization of the category of globally equivariant orthogonal spectra at the homotopical class of \emph{global equivalences} (\cite[\adef \textbf{1.2}]{SchwedeGlobal}: the homotopical category ${}_\mcal{G}\cate{Sp}$ admits a natural forgetful functor $u\colon {}_\mcal{G}\cate{Sp} \to \cate{Sp}$ which ``forgets the equivariancy'' (it is the identity on objects, and includes the class of global equivalences in the bigger class of weak equivalences of plain spectra), which has both a left and a right adjoint $u_L, u_R$, and plays the r\^ole of a $q$-functor in a recollement
\begin{equation}
\xymatrix{
	\cate{Sp}_+ \ar[r]  & \prescript{}{\mcal{G}}{\cate{Sp}} \ar[r]_u & \cate{Sp}
}
\end{equation}
where the functor $i\colon \cate{Sp}_+ \to {}_\mcal{G}\cate{Sp}$ embeds the subcategory of orthogonal spectra that are \emph{stably contractible} in the traditional, non-equivariant sense.
\begin{remark}
Since in a stable $\infty$\hyp{}category every pullback is a pushout and vice versa, 
any functor between stable $\infty$\hyp{}categories preserving either limits or colimits preserves in particular pullout diagrams. Since left adjoints and right adjoints have this property, we find
\end{remark}
\begin{proposition}[Exactness of recollement functors]\label{recoexact}\index{Recollement!exactness of ---'s functors}
Each of the functors $i,i_L,i_R,q,q_L,q_R$ in a recollement situation preserves pullout diagrams.
\end{proposition}
This simple remark will be extremely useful in view of the ``standard procedure'' for proving results in recollement theory outlined in \refbf{generalproced}.
\begin{definition}[The ($\infty$-)category $\cate{Recol}$]
A morphism between two recollements $\rec$ and $\rec'$ consists of a triple of functors $(F_0, F, F_1)$ such that the following square commutes in every part (choosing from time to time homonymous left or right adjoints):
\begin{equation}
\xymatrix{
  \D^0 \ar[r]|i \ar[d]_{F_0} & \D\ar[d]_F \ar[r]|q \ar@<-6pt>[l]\ar@<6pt>[l]& \D^1 \ar[d]^{F_1} \ar@<-6pt>[l]\ar@<6pt>[l]\\
  '{\D^0} \ar[r]|{i'} & '{\D} \ar[r]|{q'} \ar@<-6pt>[l]\ar@<6pt>[l]& '{\D^1}\ar@<-6pt>[l]\ar@<6pt>[l] \\
}
\end{equation}
This definition turns the collection of all recollement data into a $\infty$\hyp{}category denoted $\cate{Recol}$ and called the ($\infty$-)category of recollements.
\end{definition}
\begin{remark}\label{equiv.of.reco}
The natural definition of equivalence between two recollement data (all three functors $(F_0, F_{01}, F_1)$ are equivalences) has an alternative reformulation (see \cite[\abbrv{Thm} \textbf{2.5}]{parshall1988derived}) asking that only two out of three functors are equivalences; nevertheless (\emph{loc. cit.}) this must not be interpreted as a full 3-for-2 condition.

{\color{black}Equivalently,} 
 we can define this notion (see \cite[\S \textbf{1.7}]{hugel2011recollements}), asking that the essential images of the fully faithful functors $(i, q_L,q_R)$ are pairwise equivalent with those of $(i', q_L', q_R')$.
\end{remark}
\index{ttf triple@\smallcap{ttf}-triple}
We now concentrate on other equivalent ways to specify a recollement on a stable $\infty$\hyp{}category, slightly rephrasing Definition \refbf{def:recol}: first of all, \cite[\abbrv{Prop} \textbf{4.13.1}]{Hol} shows that the localization functor $q_Rq$, which is an exact localization with reflective kernel, uniquely determines the recollement datum up to equivalence; albeit of great significance as a general result, we are not interested in this perspective, and we address the interested readers to \cite{Hol} for a thorough discussion. 

Another equivalent description of a recollement, nearer to our ``torsio-centric'' approach, is via a pair of $t$\hyp{}structures on $\D$ \cite{ttftriples}:
\begin{definition}[Stable {\sc ttf} Triple]
Let $\D$ be a stable $\infty$\hyp{}category. A \emph{stable \textsc{ttf} triple} (short for \emph{torsion-torsionfree triple}) on $\D$ is a triple of full subcategories $(\mcal{X}, \mcal{Y}, \mcal{Z})$ of $\D$ such that both $(\mcal{X}, \mcal{Y})$ and $(\mcal{Y}, \mcal{Z})$ are $t$\hyp{}structures on $\D$.
\end{definition}
Notice in particular that $\D$ is reflected on $\mcal{Y}$ via a functor $R^{\mcal{Y}}$ and coreflected via a functor $S^{\mcal{Y}}$
. The whole arrangement of categories and functors is summarized in the following diagram
\begin{equation}\label{ttf}
\xymatrix@C=1.5cm@R=2mm{
 	& 	& \mcal{X} \\
 \mcal{Y}	& \cate{D}	&  \\
 	& 	& \mcal{Z}
 \ar@/_8pt/@<-3pt> "1,3";"2,2" _{i_\mcal{X}}
 \ar "2,1";"2,2" |{i_\mcal{Y}}
 \ar@<8pt> "2,2";"2,1" ^{R_\mcal{Y}}
 \ar@<-8pt> "2,2";"2,1" _{S_\mcal{Y}}
 \ar@/^8pt/@<-3pt> "2,2";"1,3" _{S_\mcal{X}}
 \ar@/_8pt/@<3pt> "2,2";"3,3" ^{R_\mcal{Z}}
 \ar@/^8pt/@<3pt> "3,3";"2,2" ^{i_\mcal{Z}}
}
\end{equation}
where $S_{\mcal{Y}}\dashv i_\mcal{Y}\dashv R_\mcal{Y}$, $i_\mcal{Z}\dashv R_\mcal{Z}$ and $S_\mcal{X}\dashv i_\mcal{X}$. 

Stable \textsc{ttf} triples are in bijection with equivalence classes of recollements, as it is recalled in \cite[\abbrv{Prop} \textbf{4.2.4}]{ttftriples}; the same bijection holds in the stable setting, \emph{mutatis mutandis}.

We conclude this introductory section with the following Lemma, which will be of capital importance all along \S\refbf{stabrecoll}: functors in a recollement jointly reflect isomorphisms.
\begin{lemma}[Joint conservativity of recollement data]\label{reflectsall}\index{Recollement!conservativity of ---'s functors}
Let $\D$ be a stable $\infty$\hyp{}category, and let $$(i,q)\colon \D^0 \lrlarrows  \D  \lrlarrows  \D^1$$
be a recollement on $\D$. Then the following conditions are equivalent for an arrow $f\in \hom(\D)$:
\begin{itemize}
\item $f$ is an isomorphism in $\D$;
\item $q(f)$ is an isomorphism in $\D^1$ and $i_R(f)$ is an isomorphism in $\D^0$;
\item $q(f)$ is an isomorphism in $\D^1$ and $i_L(f)$ is an isomorphism in $\D^0$.
\end{itemize}
In other words, the pairs of functors $\{q, i_R\}$ and $\{q, i_L\}$ \emph{jointly reflect isomorphisms}.
\end{lemma}
\begin{proof}
We only prove that if $q(f)$ and $i_L(f)$ are isomorphisms in the respective co\-do\-mains, then $f$ is an isomorphism in $\D$. We need a preparatory sub-lemma, namely that the pair $\{q, i_L\}$ reflects zero objects; the only non trivial part of this statement is that if $qD\cong 0$ in $\D^1$ and $i_LD\cong 0$ in $\D^0$, then $D\cong 0$ in $\D$, an obvious statement in view of axiom (\oldstylenums{3}) of \abbrv{Def} \refbf{def:recol}, since $qD\cong 0$ entails $D\cong i(D')$, and now $0\cong i_L(D)=i_LiD' \cong D$.

With this preliminary result, we recall that $f\colon X\to Y$ is an isomorphism if and only if $\text{fib}(f)\cong 0$, and apply the previous result, together with the fact that recollement functors preserve pullouts.

Replacing $i_L$ with $i_R$, the proof shows a similar statement about the joint reflectivity of $\{q, i_R\}$.
\end{proof}
\begin{notat}\label{veryshort}
We will often use a rather intuitive shorthand, writing $\{q, i_L\}(f)$, or $\{q, i_R\}(f)$ to both functors applied to the same arrow. For example:
\begin{itemize}
	\item Given (the left classes of) a pair of $t$\hyp{}structures $\D_{\ge 0}^0, \D_{\ge 0}^1$ we write ``$\{q, i_L\}(D)\in \D_{\ge 0}$'' (see Thm. \refbf{gluing}) to denote that the object $qD\in \D_{\ge 0}^1$ and $i_L(D)\in \D_{\ge 0}^0$; similarly for $\{q, i_R\}(D)\in \D_{< 0}$ and other combinations.
	\item Given (the left classes of) a pair of normal torsion theories $\EE_0, \EE_1$, we write ``$\{q, i_{L/R}\}(f)\in \EE$'' (see Thm. \refbf{thm:trueglued}) to denote that the arrow $f\in\hom(\D)$ is such that $qf\in \EE_1$ and $i_{L/R}(f)\in \EE_0$; similarly for $\{q, i_{L/R}\}(g)\in \MM$ and other combinations.
\end{itemize}
\end{notat}
\begin{remark}
The joint reflectivity of the recollement functors $\{q, i_L\}$ or $\{q, i_R\}$ can be seen as an analogue, in the setting of an abstract recollement, of the fact that in the geometric case of the recollement induced by a stratification $\varnothing\subset U\subset X$ one has (\cite[\textbf{2.3}]{parshall1988derived}) that a morphism of sheaves $\varphi\colon \mathscr{F}\to \mathscr{F}'$ on $X$ is uniquely determined by its restrictions $\varphi\bigr\vert_{U}$ and $\varphi\bigr\vert_{X\setminus U}$.
\end{remark}
\index{Gluing!--- of $t$\hyp{}structures}
\subsection{The classical gluing of $t$-structures.}\label{classrec}\index{t-structure@$t$\hyp{}structure!gluing of ---s|see {Gluing}}\index{.glue@$\glue$}
The main result in the classical theory of recollements is the so-called \emph{gluing theorem}, which tells us how to obtain a $t$-structure $\tee = \tee_0 \glue \tee_1$\footnote{The symbol $\smallglue$ (pron. \emph{glue}) reminds the alchemical token describing the process of \emph{amalgamation} between two or more elements (one of which is often mercury): albeit amalgamation is not recognized as a proper stage of the \emph{Magnum Opus}, several sources testify that it belongs to the alchemical tradition (see \cite[pp. \textbf{409-498}]{roth1976deutsches}).} on $\D$ starting from two $t$-structures $\tee_i$ on the categories $\D^i$ of a recollement $\rec$.
\begin{theorem}[Gluing Theorem]\label{gluing}
Consider a recollement $$\rec = (i,q)\colon \D^0 \lrlarrows  \D \lrlarrows  \D^1,$$ and let $\tee_i$ be $t$-structures on $\D^i$ for $i=0,1$; then there exists a $t$-structure on $\D$, called the \emph{gluing} of the $\tee_i$ (along the recollement $\rec$, but this specification is almost always omitted) and denoted $\tee_0\glue \tee_1$, whose classes $\big( (\D^0\glue \D^1)_{\ge 0}, (\D^0\glue \D^1)_{<0}\big)$ are given by
\begin{align}
(\D^0\glue \D^1)_{\ge 0} &= \Big\{ X\in\D\mid (q X\in \D_{\ge 0}^1)\land (i_L X\in \D_{\ge 0}^0) \Big\};\notag \\
(\D^0\glue \D^1)_{<0} &= \Big\{ X\in\D\mid ( q X\in \D_{< 0}^1)\land (i_R X\in \D_{< 0}^0) \Big\}.\label{glued}
\end{align}
\end{theorem}
\begin{remark}
Following Notation \refbf{veryshort} we have that $X\in \D_{\ge 0}$ iff $\{q, i_L\}(X)\in \D_{\ge 0}$  and $Y\in \D_{<0}$ iff $\{q, i_R\}(X)\in \D_{<0}$, which is a rather evocative statement: the left/right class of $\tee_0\glue \tee_1$ is determined by the left/right adjoint to $i$.
\end{remark}

\begin{remark}
The ``wrong way'' classes
\begin{gather}
(\D^0\glue \D^1)^\bigstar_{\ge 0} = \Big\{ X\in\D\mid (\{q,i_R\} X\in \D_{\ge 0} \Big\};\notag \\
(\D^0\glue \D^1)^\bigstar_{<0} = \Big\{ X\in\D\mid (\{q,i_L\} X\in \D_{< 0}  \Big\}.
\end{gather}
do not define a $t$-structure in general. However they do in the case the recollement situation $\rec$ is the lower part of a \emph{2-recollement}, \ie there exists a diagram of the form
\begin{equation} 
\xymatrix{
  \mathbf{C}^0  & \mathbf{C}  & \mathbf{C}^1
  \ar@{<-}@<9pt> "1,1";"1,2" ^{i_1}
  \ar@<3pt> "1,1";"1,2" |{i_2}
  \ar@{<-}@<-3pt> "1,1";"1,2" |{i_3}
  \ar@<-9pt> "1,1";"1,2" _{i_4}
  \ar@{<-}@<9pt> "1,2";"1,3" ^{q_1}
  \ar@<3pt> "1,2";"1,3" |{q_2}
  \ar@{<-}@<-3pt> "1,2";"1,3" |{q_3}
  \ar@<-9pt> "1,2";"1,3" _{q_4}
}
\end{equation}
where both
\begin{equation}
\rec_2=\xymatrix{
  \mathbf{C}^0  &\mathbf{C} & \mathbf{C}^1
  \ar|{i_2} "1,1";"1,2" 
  \ar@<8pt>^{i_3} "1,2";"1,1" 
  \ar@<-8pt>_{i_1} "1,2";"1,1" 
  \ar|{q_2} "1,2";"1,3" 
  \ar@<8pt>^{q_3} "1,3";"1,2" 
  \ar@<-8pt>_{q_1} "1,3";"1,2" 
}
\end{equation}
and
\begin{equation}
\rec_3=\xymatrix{
 \mathbf{C}^1 & \mathbf{C}  & \mathbf{C}^0
  \ar|{q_3} "1,1";"1,2" 
  \ar@<8pt>^{q_4} "1,2";"1,1" 
  \ar@<-8pt>_{q_2} "1,2";"1,1" 
  \ar|{i_3} "1,2";"1,3" 
  \ar@<8pt>^{i_4} "1,3";"1,2" 
  \ar@<-8pt>_{i_2} "1,3";"1,2" 
}
\end{equation}
are recollements, with $\rec=\rec_3$. Indeed, in this situation one has
\begin{align*}
(\D^0\glue \D^1)^\bigstar_{\ge 0} &= \Big\{ X\in\D\mid (\{q,i_R\} X\in \D_{\ge 0} \Big\}\\
&= \Big\{ X\in\mathbf{C}\mid (\{i_3,q_2\} X\in \mathbf{C}_{\ge 0} \Big\}\\
&=(\mathbf{C}^0\glue^{\rec_2} \mathbf{C}^1)_{\ge 0}.
\end{align*}
More generally, an $n$-recollement is defined as the datum of three stable $\infty$-categories $\mathbf{C}^0, \mathbf{C}, \mathbf{C}^1$ organized in a diagram
\begin{equation}\label{diag:nreco}
  \xymatrix{
    \CC^0  & \CC  & \CC^1
    \ar@<5pt> "1,1";"1,2" |{i_2}
    \ar@{<-}@<11pt> "1,1";"1,2" ^{i_1}
    \ar@{<-}@<-3pt> "1,1";"1,2" |{i_3}
    \ar@<-10pt>@{}|\vdots "1,1";"1,2" 
    \ar@{<-}@<-15pt> "1,1";"1,2" _{i_{n+2}}
    \ar@<5pt> "1,2";"1,3" |{q_2}
    \ar@{<-}@<11pt> "1,2";"1,3" ^{q_1}
    \ar@{<-}@<-3pt> "1,2";"1,3" |{q_3}
    \ar@<-10pt>@{}|\vdots "1,2";"1,3" 
    \ar@{<-}@<-15pt> "1,2";"1,3" _{q_{n+2}}
  }
\end{equation}with $n+2$ functors on each edge, such that every consecutive three functors form recollements  $\rec_{2k} = (i_{2k}, q_{2k})$, $\rec_{2h+1} = (q_{2h+1}, i_{2h+1})$, for $k=1, \dots, n-1$, $h=1, \dots, n-2$, see \cite[Def. \textbf{2}]{nrecol1}.  Applications of this formalism to derived categories of algebras, investigating the relationships between the recollements of derived categories
and the Gorenstein properties of these algebras, can be found in \cite{nrecol1,nrecol2}.
\end{remark}

\begin{notat}
\label{notat:recdec}
It is worth to notice that $\D^0\glue \D^1$ \index{.D^0glueD^1@$\D^0\glue \D^1$}has no real meaning as a category; this is only an intuitive shorthand to denote the pair $(\D,\tee_0\glue\tee_1)$; \marginnote{\dbend} more explicitly, it is a shorthand to denote the following situation:
\begin{quote}
The stable $\infty$-category $\D$ fits into a recollement $(i,q)\colon \D^0\lrlarrows \D \lrlarrows \D^1$, $t$-structures on $\D^0$ and $\D^1$ have been chosen, and $\D$ is endowed with the glued $t$-structure $\tee_0\glue \tee_1$.
\end{quote}
\end{notat}
A proof of the gluing theorem in the classical setting of triangulated categories can be found in \cite[Thm. \textbf{7.2.2}]{Banagl} or in the standard reference \cite{BBDPervers}. We briefly sketch the argument given in \cite{Banagl} as we will need it in the torsio-centric reformulation of the gluing theorem. 
\begin{proof}[Proof of Thm. \protect{\ref{gluing}}]
We begin showing the way in which every $X\in\D$ fits into a fiber sequence $SX\to X\to RX$ where $SX\in (\D^0\glue \D^1)_{\ge 0}, RX\in (\D^0\glue \D^1)_{<0}$. \index{Gluing!co\fshyp{}reflection of a ---}
Let $\fF_i$ denote the normal torsion theory on $\D^i$, inducing the $t$\hyp{}structure $\tee_i$; let $\eta_1 \colon qX\to R_1 qX$ be the arrow in the fiber sequence
\begin{equation}\label{fibseqofq}
S_1 qX \xto{\epsilon_1} qX\xto{\eta_1} R_1 qX
\end{equation}
obtained thanks to $\fF_1$; let $\hat\eta$ be its \emph{mate} $X\to q_R R_1 qX$ in $\D$ under the adjunction $q\dashv q_R$, and let $WX=\fib(\hat \eta)$. 

Now, consider $i_L WX$ in the fiber sequence 
\[
S_0 i_L WX \xto{\sigma_0} i_L WX \xto{\theta_0} R_0 i_LWX\]
induced by $\fF_0$ on $\D_0$, and its mate $\hat \theta\colon WX\to i R_0 i_LWX$; take its fiber $SX$, and the object $RX$ defined as the pushout of $i R_0 i_L WX \xot{\hat\theta}
 WX \to X$.

To prove that these two objects are the candidate co/trun\-ca\-tion we consider the diagram
\[\label{star}
\xymatrix{
  SX  & WX  & X \\
  0 & i R_0 i_L WX  & RX \\
    & 0 & q_R R_1 qX
  \ar "1,1";"1,2" 
  \ar "1,1";"2,1" 
  \ar "1,2";"1,3" 
  \ar "1,2";"2,2" _{\hat\theta}
  \ar@{.>}@/^1.4pc/ "1,3";"3,3" ^{\hat\eta}
  \ar "1,3";"2,3" 
  \ar "2,1";"2,2" 
  \ar "2,2";"2,3" 
  \ar "2,2";"3,2" 
  \ar "2,3";"3,3" 
  \ar "3,2";"3,3" 
}\]
where all the mentioned objects fit, and where every square is a pullout. We have to prove that $SX\in (\D^0\glue \D^1)_{\ge 0}$ and $RX\in (\D^0\glue \D^1)_{<0}$. To do this, apply the functors $q, i_L, i_R$ to (\refbf{star}), obtaining the following diagram of pullout squares (recall the e\-xact\-ness properties of the recollement functors, stated in \aprop \refbf{recoexact}):
\[
\scalebox{.9}{%
\xymatrix@C=6mm{
  qSX & qWX & qX \\
  0 & 0 & q RX \\
    & 0 & R_1 qX
  \ar^\sim "1,1";"1,2" 
  \ar "1,1";"2,1" 
  \ar "1,2";"1,3" 
  \ar "1,2";"2,2" 
  \ar "1,3";"2,3" 
  \ar@{=} "2,1";"2,2" 
  \ar "2,2";"2,3" 
  \ar@{=} "2,2";"3,2" 
  \ar "2,3";"3,3" 
  \ar "3,2";"3,3" 
}
\xymatrix@C=6mm{
  i_L SX  & i_L WX  & i_L X \\
  0 & R_0 i_L WX  & i_L RX \\
    & 0 & i_L q_R R_1 qX
  \ar@{} "1,1";"2,2" |{\uno}
  \ar "1,1";"1,2" 
  \ar "1,1";"2,1" 
  \ar "1,2";"1,3" 
  \ar "1,2";"2,2" 
  \ar "1,3";"2,3" 
  \ar "2,1";"2,2" 
  \ar "2,2";"2,3" 
  \ar "2,2";"3,2" 
  \ar "2,3";"3,3" 
  \ar "3,2";"3,3" 
}
\xymatrix@C=6mm{
  i_R SX  & i_R WX  & i_R X \\
  0 & R_0 i_L WX  & i_R RX \\
    & 0 & 0
  \ar "1,1";"1,2" 
  \ar "1,1";"2,1" 
  \ar "1,2";"1,3" 
  \ar "1,2";"2,2" 
  \ar "1,3";"2,3" 
  \ar "2,1";"2,2" 
  \ar^\sim "2,2";"2,3" 
  \ar "2,2";"3,2" 
  \ar "2,3";"3,3" 
  \ar@{=} "3,2";"3,3" 
}}\notag
\]
where we took into account the relations $qi=0, i_Rq_R = 0 = i_L q_L$. We find that
\begin{itemize}
\item $qSX\cong qWX\cong S_1 qX\in \D_{\ge 0}^1$, since $0\to S_1 qX$ lies in $\MM_1$, and $qRX\cong R_1 qX\in\D_{<0}^1$;
\item $i_L SX\cong S_0 i_L WX \in \D_{\ge 0}^0$, by the pullout square \uno;
\item $i_R RX\cong R_0 i_L WX \in \D_{<0}^0$.
\end{itemize}
It remains to show that the two classes $\D_{\ge 0},\D_{<0}$ are orthogonal; to see this, suppose that $X\in \D_{\ge 0}$ and $Y\in \D_{<0}$. We consider the fiber sequence $ii_R Y\to Y\to q_R q Y$ of axiom $(\oldstylenums{4})$ in \abbrv{Def} \refbf{def:recol}, to obtain (applying the homological functor $\D(X, -)$) 
\begin{equation}
\xymatrix@R=2mm{
  \D(X, ii_R Y)	& \D(X, Y)	& \D(X,q_R q Y) \\
  \D(i_LX, i_R Y)	& 	& \D(qX, qY) \\
  0	& 	& 0
  \ar@{=} "1,1";"2,1" 
  \ar "1,1";"1,2" 
  \ar "1,2";"1,3" 
  \ar@{=} "1,3";"2,3" 
  \ar@{=} "2,1";"3,1" 
  \ar@{=} "2,3";"3,3" 
}
\end{equation}and we conclude, thanks to the exactness of this sequence.
\end{proof}
\begin{remark}
Strictly speaking, the domain of definition of the gluing operation $\glue$ is the set of triples $(\tee_0, \tee_1, \rec)$ where $(\tee_0,\tee_1)\in \ts (\D^0)\times \ts (\D^1)$ and $\rec=(i,q)$ is a recollement $\D^0\lrlarrows \D \lrlarrows \D^1$, but unless this (rather stodgy) distinction is strictly necessary we will adopt an obvious abuse of notation.
\end{remark}
\begin{remark}[A standard technique]\label{generalproced}
The procedure outlined above is in some sense paradigmatic, and it's worth to trace it out as an abstract way to deduce properties about objects and arrows fitting in a diagram like (\refbf{star}). This algorithm will be our primary technique to prove statements in the ``torsio-centric'' formulation of recollements:
\begin{itemize}
\item We start with a particular diagram, like for example (\refbf{star}) or (\refbf{twostar}) below; our aim is to prove that a property (being invertible, being the zero map, lying in a distinguished class of arrows, etc.\@\xspace) is true for an arrow $h$ in this diagram.
\item We apply (possibly only some of) the recollement functors to the diagram, and we deduce that $h$ has the above property from
\begin{itemize}
\item The recollement relations between the functors (\abbrv{Def} \refbf{def:recol});
\item The exactness of the recollement functors (\abbrv{Prop} \refbf{recoexact});
\item The joint reflectivity of the pairs $\{q, i_L\}$ and $\{q, i_R\}$ (Lemma \refbf{reflectsall});
\end{itemize}
\end{itemize}
\end{remark}

\section{Stable Recollements.}\label{stabrecoll}
\renewcommand{\textflush}{flushright}.
\epigraph{%
    \cjRL{	way*a.ha:loM w:hin*eh sul*AM mu.s*Ab 
		'ar:.sAh w:ro'+sO mag*i`a ha+s*AmAy:mAh 
		w:hin*eh mal:'a:key 'E:lohiym `oliym 
		w:yor:diym b*wo;}}
{\cite{BHS}, \textsc{Genesis} \oldstylenums{28}:\oldstylenums{12}}
\renewcommand{\textflush}{flushleft}.
\index{Gluing!Jacobladder@``Jacob ladder''}
\subsection{The Jacob's ladder: building co\fshyp{}reflections.}
The above procedure to build the functors $R,S$ depends on several choices (we forget half of the fiber sequence $S_1 q X\to qX\to R_1 qX$) and it doesn't seem independent from these choices, at least at first sight. 

The scope of this first subsection is to show that this apparent asymmetry arises only because we are hiding half of the construction, taking into account only half of the fiber sequence (\refbf{fibseqofq}). Given an object $X\in \D$ a dual argument yields \emph{another} way to construct a fiber sequence
\begin{equation}
	S' X\to X\to R' X
\end{equation}
out of the recollement data, which is naturally isomorphic to the former $SX\to X\to RX$. 

We briefly sketch how this dualization process goes: starting from the coreflection arrow $\epsilon_1 \colon S_1qX\to qX$, taking its mate $q_LS_1 qX\to X$ under the adjunction $q_L\dashv q$, and reasoning about its cofiber we can build a diagram which is dual to the former one, and where every square is a pullout:
\[\label{twostar}
\scalebox{.9}{\xymatrix{
                q_L S_1 q X & S'X & X \\
                0 & i S_0 i_R KX  & KX \\
                  & 0 & R'X
                \ar "1,1";"1,2" 
                \ar "1,1";"2,1" 
                \ar "1,2";"1,3" 
                \ar "1,2";"2,2" 
                \ar "1,3";"2,3" 
                \ar "2,1";"2,2" 
                \ar "2,2";"2,3" 
                \ar "2,2";"3,2" 
                \ar "2,3";"3,3" 
                \ar "3,2";"3,3" 
              }}
\]
\begin{proposition}[The Jacob's ladder]\label{thejacbo}
The two squares of the previous constructions fit into a ``ladder'' induced by canonical isomorphisms $SX\cong S'X, RX\cong R'X$; the construction is functorial in $X$. The ``Jacob's ladder'' is the following diagram:
\[\label{ladder:objs}
\scalebox{.9}{\xymatrix{
                q_L S_1 qX  & SX  & WX  & X \\
                0 & iS_0 i_R KX & CX  & KX \\
                  & 0 & i R_0 i_L WX  & RX \\
                  &   & 0 & q_R R_1 q X
                \ar "1,1";"1,2" 
                \ar "1,1";"2,1" 
                \ar "1,2";"1,3" 
                \ar "1,2";"2,2" 
                \ar "1,3";"1,4" 
                \ar "1,3";"2,3" 
                \ar "1,4";"2,4" 
                \ar "2,1";"2,2" 
                \ar "2,2";"2,3" 
                \ar "2,2";"3,2" 
                \ar "2,3";"2,4" 
                \ar "2,3";"3,3" 
                \ar "2,4";"3,4" 
                \ar "3,2";"3,3" 
                \ar "3,3";"3,4" 
                \ar "3,3";"4,3" 
                \ar "3,4";"4,4" 
                \ar "4,3";"4,4" 
              }}
\]
\end{proposition}
\begin{proof}
It suffices to prove that both $SX, S'X$ lie in $\D_{\ge 0}$ and both $RX, R'X$ lie in $\D_{\le 0}$; given this, we can appeal (a suitable stable $\infty$-categorical version of) \cite[\abbrv{Prop} \textbf{1.1.9}]{BBDPervers} which asserts the functoriality of the truncation functors, \ie that when the same object $X$ fits into \emph{two} fiber sequences arising from the same normal torsion theory, then there exist the desired isomorphisms.\footnote{In a torsio-centric perspective, this follows from the uniqueness of the factorization of a morphism with respect to the normal torsion theory having reflection $R$ and coreflection $S$.}

The procedure showing this is actually the same remarked in \refbf{generalproced}: we apply $q,i_L, i_R$ to the diagram (\refbf{twostar}) and we exploit exactness of the recollement functors to find pullout diagrams showing that $R'X\in \D_{<0}$ and $S'X\in \D_{\ge 0}$.

Once these isomorphisms have been found, it remains only to glue the two sub-diagrams
$$\scalebox{.8}{
\xymatrix@R=7mm@C=4mm{
 \gray{q_L S_1 qX}  & SX  & WX  & X \\
 \gray{0} & \gray{iS_0 i_R KX}  & \gray{CX} & \gray{KX} \\
  & 0 & i R_0 i_L WX  & RX \\
  &   & 0 & q_R R_1 q X
 \ar@{.>}@[semilightgray] "1,1";"1,2" 
 \ar@{.>}@[semilightgray] "1,1";"2,1" 
 \ar "1,2";"1,3" 
 \ar@{-} "1,2";"2,2" 
 \ar "1,3";"1,4" 
 \ar@{-} "1,3";"2,3" 
 \ar@{-} "1,4";"2,4" 
 \ar@{.>}@[semilightgray] "2,1";"2,2" 
 \ar@{.>}@[semilightgray] "2,2";"2,3" 
 \ar "2,2";"3,2" 
 \ar@{.>}@[semilightgray] "2,3";"2,4" 
 \ar "2,3";"3,3" 
 \ar "2,4";"3,4" 
 \ar "3,2";"3,3" 
 \ar "3,3";"3,4" 
 \ar "3,3";"4,3" 
 \ar "3,4";"4,4" 
 \ar "4,3";"4,4" 
}
  \xymatrix@R=7mm@C=4mm{
    q_L S_1 qX  & S'X & \gray{WX} & X \\
    0 & iS_0 i_R KX & \gray{CX} & KX \\
      & 0 & \gray{i R_0 i_L WX} & R'X \\
      &   & \gray{0}  & \gray{q_R R_1 q X}
    \ar "1,1";"1,2" 
    \ar "1,1";"2,1" 
    \ar@{-} "1,2";"1,3" 
    \ar "1,2";"2,2" 
    \ar "1,3";"1,4" 
    \ar@{.>}@[semilightgray] "1,3";"2,3" 
    \ar "1,4";"2,4" 
    \ar "2,1";"2,2" 
    \ar@{-} "2,2";"2,3" 
    \ar "2,2";"3,2" 
    \ar "2,3";"2,4" 
    \ar@{.>}@[semilightgray] "2,3";"3,3" 
    \ar "2,4";"3,4" 
    \ar@{-} "3,2";"3,3" 
    \ar "3,3";"3,4" 
    \ar@{.>}@[semilightgray] "3,3";"4,3" 
    \ar@{.>}@[semilightgray] "3,4";"4,4" 
    \ar@{.>}@[semilightgray] "4,3";"4,4" 
  }
}$$
to obtain the ladder. Now, this construction is obtained by taking into account the fiber sequence $S_1 q X\to q X\to R_1 q X$ as a whole, and since this latter object is uniquely determined up to isomorphism, we obtain a diagram of endofunctors
\[\label{ladder:arrows}
  \scalebox{.8}{\xymatrix{
                  q_L S_1 q & S & W & 1 \\
                  0 & iS_0 i_R K  & C & K \\
                    & 0 & i R_0 i_L W & R \\
                    &   & 0 & q_R R_1 q
                  \ar "1,1";"1,2" 
                  \ar "1,1";"2,1" 
                  \ar "1,2";"1,3" 
                  \ar "1,2";"2,2" 
                  \ar "1,3";"1,4" 
                  \ar "1,3";"2,3" 
                  \ar "1,4";"2,4" 
                  \ar "2,1";"2,2" 
                  \ar "2,2";"2,3" 
                  \ar "2,2";"3,2" 
                  \ar "2,3";"2,4" 
                  \ar "2,3";"3,3" 
                  \ar "2,4";"3,4" 
                  \ar "3,2";"3,3" 
                  \ar "3,3";"3,4" 
                  \ar "3,3";"4,3" 
                  \ar "3,4";"4,4" 
                  \ar "4,3";"4,4" 
                }}
\]where every square is a pullout (again giving to a category of functors the obvious stable structure \cite[\abbrv{Prop} \textbf{1.1.3.1}]{LurieHA}), and where the functorial nature of $W$, $K$ and $C$ is a consequence of their construction. Notice also that this latter diagram of functors uses homogeneously all the recollement functors, and that it is ``symmetric'' with respect to the antidiagonal (it switches left and right adjoints, as well as reflections and coreflections).
\end{proof}
The functors $S,R$ are the co/truncations for the recoll\'ee $t$-structure, and 
the normality of the torsion theory is witnessed by the pullout subdiagram
\[\label{normalite}
  \scalebox{.8}{\xymatrix@!R=7mm{
                  SX  & \gray{WX} & X \\
                  \gray{iS_0 i_RKX} & \gray{CX} & \gray{KX} \\
                  0 & \gray{i R_0 i_L WX} & RX.
                  \ar@{}|(.3)\lrcorner "1,1";"2,2" 
                  \ar@{-} "1,1";"1,2" 
                  \ar@{-} "1,1";"2,1" 
                  \ar "1,2";"1,3" 
                  \ar@{.>}@[semilightgray] "1,2";"2,2" 
                  \ar@{-} "1,3";"2,3" 
                  \ar@{.>}@[semilightgray] "2,1";"2,2" 
                  \ar "2,1";"3,1" 
                  \ar@{.>}@[semilightgray] "2,2";"2,3" 
                  \ar@{.>}@[semilightgray] "2,2";"3,2" 
                  \ar "2,3";"3,3" 
                  \ar@{-} "3,1";"3,2" 
                  \ar "3,2";"3,3" 
                  \ar@{}|(.3)\ulcorner "3,3";"2,2" 
                }}
\]\begin{notat}
From now on, we will always refer to the diagram above as ``the Jacob ladder'' of an object $X\in\D$, and/or to the diagram induced by a morphism $f\colon X\to Y$ between the ladder of the domain and the codomain, \ie to three-dimensional diagrams like
\[\label{functorial-ladder}
\scalebox{.8}{ 
\xymatrix@R=4.5mm@C=2mm{
    & q_L S_1 q Y &   & \textcolor{black}{SY} &   & WY  &   & \textcolor{black}{Y} \\
  q_L S_1 qX  &   & \textcolor{black}{SX} &   & WX  &   & \textcolor{black}{X} \\
    &   &   &   &   &   &   & KY \\
  0 &   & i S_0 i_R KX  &   & CX  &   & KX \\
    &   &   &   &   &   &   & \textcolor{black}{RY} \\
    &   & 0 &   & i R_0 i_L WX  &   & \textcolor{black}{RX} \\
    &   &   &   &   &   &   & q_R R_1 qY \\
    &   &   &   & 0 &   & q_R R_1 qX
  \ar "1,2";"1,4" 
  \ar "1,4";"1,6" 
  \ar "1,6";"1,8" 
  \ar "1,8";"3,8" 
  \ar "2,1";"1,2" 
  \ar "2,1";"2,3" 
  \ar "2,1";"4,1" 
  \ar "2,3";"1,4" ^{Sf}
  \ar "2,3";"2,5" 
  \ar "2,3";"4,3" 
  \ar "2,5";"1,6" ^{Wf}
  \ar "2,5";"2,7" 
  \ar "2,5";"4,5" 
  \ar "2,7";"4,7" 
  \ar^f "2,7";"1,8" 
  \ar "3,8";"5,8" 
  \ar "4,1";"4,3" 
  \ar "4,3";"6,3" 
  \ar "4,3";"4,5" 
  \ar "4,5";"6,5" 
  \ar "4,5";"4,7" 
  \ar "4,7";"6,7" 
  \ar "4,7";"3,8" ^{Kf}
  \ar "5,8";"7,8" 
  \ar "6,3";"6,5" 
  \ar "6,5";"6,7" 
  \ar "6,5";"8,5" 
  \ar "6,7";"8,7" 
  \ar "6,7";"5,8" ^{Rf}
  \ar "8,5";"8,7" 
  \ar "8,7";"7,8" 
}}
\]
\end{notat}
\subsection{The {\sc ntt} of a recollement.}
\index{Gluing!\smallcap{ntt} of a ---}
Throughout this subsection we outline the torsio-centric translation of the classical results recalled above. In particular we give an explicit definition of the $\glue$ operation when it has been ``transported'' to the set of normal torsion theories, independent from its characterization in terms of the pairs a\-isle-co\-a\-isle of the two $t$\hyp{}structures. From now on we assume given a recollement
$$
\xymatrix{
  \D^0	& \D	& \D^1.
  \ar|i "1,1";"1,2" 
  \ar@<8pt> "1,2";"1,1" ^{i_L}
  \ar@<-8pt> "1,2";"1,1" _{i_R}
  \ar|q "1,2";"1,3" 
  \ar@<8pt> "1,3";"1,2" ^{q_L}
  \ar@<-8pt> "1,3";"1,2" _{q_R}
}$$
Given $t$\hyp{}structures $\tee_i\in  \ts (\D^i)$, in view of our ``Rosetta stone'' theorem \cite{FL0}, there exist normal torsion theories $\fF_i=(\EE_i, \MM_i)$ on $\D^i$ such that $(\D^i_{\ge 0}, \D^i_{<0})$ are the classes $(0/\EE_i,\MM_i/0)$ of torsion and torsion-free objects of $\D^i$, for $i=0,1$; an object $X$ lies in $(\D^0\glue \D^1)_{\ge 0}$ if and only if $qX\in \EE_1$ and $i_LX\in \EE_0$\footnote{Thanks to the Sator lemma we are allowed to use ``$X\in \K$'' as a shorthand to denote that either the initial arrow $\var{0}{X}$ or the terminal arrow $\var{X}{0}$ lie in a 3-for-2 class $\K \subset \hom(\CC)$. From now on we will adopt this notation.}, and similarly an object $Y$ lies in $\D_{\le 0}$ if and only if $qY\in \MM_1$ and $i_RY\in \MM_0$.
\begin{remark}\label{must-itself-come}
The $t$\hyp{}structure $\tee = \tee_0\glue \tee_1$ on $\D$ must itself come from a normal torsion theory which we denote $\fF_0\glue \fF_1$ on $\D$, so that $\big((\D^0\glue \D^1)_{\ge 0},(\D^0\glue \D^1)_{<0}\big)=\big(0/(\EE_0\glue\EE_1), (\MM_0\glue \MM_1)/0\big)$; in other words the following three conditions are equivalent for an object $X\in\D$:
\begin{itemize}
\item $X$ lies in $(\D^0\glue \D^1)_{\ge 0}$;
\item $X$ lies in $\EE_0\glue\EE_1$, \ie $RX\cong 0$ in the notation of (\refbf{normalite});
\item $\{q, i_L\}(X)\in \EE$, following Notation \refbf{veryshort}.
\end{itemize}
\end{remark}
We now aim to a torsio-centric characterization of the classes $(\EE_0\glue\EE_1, \MM_0\glue \MM_1)$, relying on the factorization properties of $(\EE_i, \MM_i)$ alone: since we proved \abbrv{Thm}\refbf{gluing} above, there must be a normal torsion theory $\fF_0\glue \fF_1 = \big(\EE_0\glue\EE_1, \MM_0\glue\MM_1\big)$ inducing $\tee_0\glue\tee_1$ as $\big(0/(\EE_0\glue\EE_1), (\MM_0\glue\MM_1)/0\big)$: in other words,
\begin{quote}
$\fF_0\glue \fF_1$ is the (unique) normal torsion theory whose torsion\fshyp{}torsionfree classes are $\big((\D^0\glue \D^1)_{\ge 0},(\D^0\glue \D^1)_{<0}\big)$ of \abbrv{Thm}\refbf{gluing},
\end{quote}
Clearly this is only an application of our ``Rosetta stone'' theorem, so in some sense this result is ``tautological''. But there are at least two reasons to concentrate in ``proving again'' \abbrv{Thm}\refbf{gluing} from a torsio-centric perspective:
\begin{itemize}
\item The construction offered by the Rosetta stone is rather indirect, and only appropriate to show formal statements about the factorization system $\fF(\tee)$ induced by a $t$\hyp{}structure;
\item In a stable setting, the torsio-centric point of view, using factorization systems, is more primitive and more natural than the classical one using 1-categorical arguments (\ie, $t$\hyp{}structures $\tee$ on the homotopy category of a stable $\D$ are induced by normal torsion theories in $\D$; in the quotient process one loses important informations about $\tee$).
\end{itemize}
Both these reasons lead us to adopt a ``constructive'' point of view, giving an explicit characterization of $\fF_0\glue \fF_1$ which relies on properties of the factorization systems $\fF_0$, $\fF_1$ alone, independent from triangulated categorical arguments. 

In the following section we will discuss the structure and properties of the factorization system $\fF_0\glue \fF_1$, concentrating on a self-contained and categorically well motivated construction of the classes $\EE_0\glue\EE_1$ and $\MM_0\glue\MM_1$ starting from an obvious \emph{ansatz} which follows Remark \refbf{must-itself-come}.

The discussion above, and in particular the fact that an initial\fshyp{}terminal arrow $0\leftrightarrows X$ lies in $\EE_0\glue\EE_1$ if and only if $\{q, i_L\}(X)\in \EE$, suggests that we define 
$\EE_0\glue\EE_1 = \big\{f\in\hom(\D)\mid 
\{q, i_L\}(f)\in \EE 
\big\}$ and $\MM_0\glue\MM_1 = \big\{g\in\hom(\D)\mid 
\{q, i_R\}(g) \in \MM
\big\}$. Actually it turns out that this guess is not far to be correct: the correct classes are indeed given by the following: 
\begin{theorem}
\label{thm:trueglued}
Let $\D$ be a stable $\infty$-category, in a recollement 
$$
(i,q)\colon \D^0 \lrlarrows  \D  \lrlarrows  \D^1,
$$
and let $\tee_i$ be a $t$-structure on $\D^i$. Then the recoll\'ee $t$-structure $\tee_0\glue \tee_1$ is induced by the normal torsion theory $(\EE_0\glue\EE_1,\MM_0\glue\MM_1)$ with classes
\begin{gather}
\EE_0\glue\EE_1 = \big\{f\in\hom(\D)\mid 
\{q, i_LW\}(f)\in \EE 
\big\};\label{BBDleftclass}\\
\MM_0\glue\MM_1 = \big\{g\in\hom(\D)\mid 
\{q, i_RK\}(g) \in \MM
\big\}.
\end{gather}
\end{theorem}
\begin{proof}
We only need to prove the statement for $\EE_0\glue\EE_1$, since the statement for $\MM_0\glue\MM_1$ is completely specular. Thanks to the discussion in section \S\refbf{classical}, an arrow $f\in\hom(\D)$ lies in $\EE_0\glue \EE_1$ if and only if $Rf$ (as constructed in the Jacob ladder (\refbf{functorial-ladder})) is an isomorphism in $\D$, so we are left to prove that, given $f\in\hom(\D)$:
\begin{equation}
\text{$Rf$ is an isomorphism in $\D$}\quad\Leftrightarrow\quad 
\text{$\{q, i_LW\}(f)\in \EE$}.
\end{equation}
Equivalently, we have to prove that
\begin{equation}
\text{$Rf$ is an isomorphism}\quad\Leftrightarrow\quad 
\text{$\{R_1q, R_0i_LW\}(f)$ are isomorphisms}.
\end{equation}
We begin by showing that if $\{ R_1q, R_0i_LW\}(f)$ are isomorphisms, then also $Rf$ is an isomorphism.
By the joint conservativity of the recollement data (Lemma \ref{reflectsall}) we need to prove that  if $\{ R_1q, R_0i_LW\}(f)$ are isomorphisms, then both $qRf$ and $i_LRf$ are isomorphisms.
Apply the functor $q$ to the Jacob ladder (\refbf{functorial-ladder}), to obtain

\begin{equation}\label{q-of-Jacob}
\scalebox{.8}{ 
\xymatrix@R=2mm@C=3mm{
    & S_1 q Y &   & q SY  &   & q WY  &   & q Y \\
  S_1 qX  &   & q SX  &   & q WX  &   & q X \\
    &   &   &   &   &   &   & q KY \\
  0 &   & 0 &   & q CX  &   & q KX \\
    &   &   &   &   &   &   & q RY \\
    &   & 0 &   & 0 &   & qRX \\
    &   &   &   &   &   &   & R_1 qY \\
    &   &   &   & 0 &   & R_1 qX
  \ar^\sim "1,2";"1,4" 
  \ar^\sim "1,4";"1,6" 
  \ar "1,6";"1,8" 
  \ar "1,8";"3,8" 
  \ar "2,1";"1,2" 
  \ar^\sim "2,1";"2,3" 
  \ar "2,1";"4,1" 
  \ar "2,3";"1,4" 
  \ar^\sim "2,3";"2,5" 
  \ar "2,3";"4,3" 
  \ar "2,5";"1,6" 
  \ar "2,5";"2,7" 
  \ar "2,5";"4,5" 
  \ar "2,7";"4,7" 
  \ar "2,7";"1,8" 
  \ar^\wr "3,8";"5,8" 
  \ar@{=} "4,1";"4,3" 
  \ar@{=} "4,3";"6,3" 
  \ar^\sim "4,3";"4,5" 
  \ar^\wr "4,5";"6,5" 
  \ar "4,5";"4,7" 
  \ar^\wr "4,7";"6,7" 
  \ar "4,7";"3,8" 
  \ar^\wr "5,8";"7,8" 
  \ar@{=} "6,3";"6,5" 
  \ar "6,5";"6,7" 
  \ar@{=} "6,5";"8,5" 
  \ar^\wr "6,7";"8,7" 
  \ar "6,7";"5,8" 
  \ar "8,5";"8,7" 
  \ar "8,7";"7,8" 
}}
\end{equation} 
Hence $qRf$ is an isomorphism, since it fits into the square
\begin{equation}
  \xymatrix{
    q R X & q R Y \\
    R_1 q X & R_1 q Y.
    \ar "1,1";"1,2" 
    \ar_\wr "1,1";"2,1" 
    \ar^\wr "1,2";"2,2" 
    \ar_\sim "2,1";"2,2" 
  }
\end{equation} 
Now apply the functor $i_L$ to the Jacob ladder, obtaining 
\begin{equation} 
\label{diag1bis}
\scalebox{.8}{
\xymatrix@R=2mm@C=3mm{
    & 0 &   & i_L SY  &   & i_L WY  &   & i_L Y \\
  0 &   & i_L SX  &   & i_L WX  &   & i_L X \\
    &   &   &   &   &   &   & i_L KY \\
  0 &   & S_0 i_L KX  &   & i_L CX  &   & i_L KX \\
    &   &   &   &   &   &   & i_L RY \\
    &   & 0 &   & R_0 i_L WX  &   & i_L RX \\
    &   &   &   &   &   &   &i_Lq_R R_1 qY \\
    &   &   &   & 0 &   & i_Lq_R R_1 qX
  \ar "1,2";"1,4" 
  \ar "1,4";"1,6" 
  \ar "1,6";"1,8" 
  \ar_\wr "1,8";"3,8" 
  \ar "2,1";"2,3"
   \ar "2,1";"1,2" 
  \ar@{=} "2,1";"4,1" 
  \ar "2,3";"2,5" 
  \ar_\wr "2,3";"4,3" 
  \ar "2,3";"1,4" 
  \ar "2,5";"2,7" 
  \ar_\wr "2,5";"4,5" 
  \ar "2,5";"1,6" 
  \ar@{} "2,7";"3,8" |{\uno}
  \ar_\wr "2,7";"4,7" 
  \ar "2,7";"1,8" 
  \ar@{->>} "3,8";"5,8" 
  \ar "4,1";"4,3" 
  \ar "4,3";"4,5" 
  \ar "4,3";"6,3" 
  \ar "4,5";"4,7" 
  \ar@{->>} "4,5";"6,5" 
  \ar@{} "4,7";"5,8" |{\due}
  \ar@{->>} "4,7";"6,7" 
  \ar "4,7";"3,8" 
  \ar "5,8";"7,8" 
  \ar "6,3";"6,5" 
  \ar "6,5";"6,7" 
  \ar "6,5";"8,5" 
  \ar "6,7";"5,8" 
  \ar "6,7";"8,7" 
  \ar "8,5";"8,7" 
  \ar "8,7";"7,8"
}}
\end{equation}
As noticed above, $R_1q f$ is an isomorphism, so also 
$i_Lq_RR_1q f$ is an isomorphism. Then $i_LRf$ is an isomorphism by the five-lemma applied to the morphism of fiber sequences
\begin{equation}\label{morphism-of-fiber-sequences}
\scalebox{.8}{
\xymatrix@R=2mm@C=3mm{
            &   & R_0 i_L WY  & i_L RY \\
           R_0 i_L WX &   & i_L RX \\
            &   &   &i_Lq_R R_1 qY \\
           0  &   & i_Lq_R R_1 qX
  \ar "2,3";"4,3" 
   \ar "2,1";"2,3" 
  \ar "2,1";"4,1" 
  \ar "4,1";"4,3" 
    \ar "2,3";"1,4" 
   \ar "4,3";"3,4" 
  \ar "1,4";"3,4" 
  \ar "2,1";"1,3"
   \ar "1,3";"1,4" 
 }}
\end{equation}

Vice versa: assuming $Rf$ is an isomorphism in $\D$, we want to prove that $\{R_1 q , R_0i_LW \}(f)$ are isomorphisms.  Diagram (\refbf{q-of-Jacob}) gives directly that $R_1 qf$ is an isomorphism, since the square
\begin{equation}
  \xymatrix{
    q RX  & q RY \\
    R_1 q X & R_1 q Y
    \ar^\sim "1,1";"1,2" 
    \ar_\wr "1,1";"2,1" 
    \ar^\wr "1,2";"2,2" 
    \ar "2,1";"2,2" 
  }
\end{equation} is commutative.
Then, from diagram (\ref{morphism-of-fiber-sequences}) we see that, since both $i_Lq_RR_1qf$ and $i_LRf$ are isomorphisms, so is also $R_0i_LWf$. 
\end{proof}
\begin{remark}
From the sub-diagram
\begin{equation} 
\xymatrix{
  i_L X & i_L Y \\
  i_L KX  & i_L KY \\
  i_L RX  & i_L RY
  \ar@{} "1,1";"2,2" |{\uno}
  \ar "1,1";"1,2" ^{i_L f}
  \ar_\wr "1,1";"2,1" 
  \ar^\wr "1,2";"2,2" 
  \ar@{} "2,1";"3,2" |{\due}
  \ar@{->>} "2,1";"2,2" 
  \ar@{->>} "2,1";"3,1" 
  \ar@{->>}@{->>} "2,2";"3,2" 
  \ar_\sim "3,1";"3,2" 
}
\end{equation} of diagram (\refbf{diag1bis}) one deduces that if $Rf$ is an isomorphism, then $i_L f\in \EE_0$, by the 3-for-2 closure property of $\EE_0$. This mean that $\{ q, i_LW\}(f)\in \EE$ implies that $\{ q, i_L\}(f)\in \EE$. The converse implication has no reason to be true in general. However it is true for terminal (or initial) morphisms. Namely, from the Rosetta stone one has that $X\in \EE_0\glue\EE_1$ if and only if $X\in (\D^0\glue \D^1)_{\ge 0}$, and so if and only if $\{q, i_L\}(X)\in \EE$. On the other hand, $X\in \EE_0\glue\EE_1$ if and only if $\{q, i_LW\}(X)\in \EE$. The fact that the condition $\{q, i_L\}(X)\in \D_{\geq 0}$ is equivalent to the condition $\{q, i_LW\}(X)\in \D_{\geq 0}$ can actually be easily checked directly. Namely, if $qX\in \D^1_{\geq 0}$, then $q_RR_1qX=0$ and so $X=WX$ in this case. Specular considerations apply to the right class $\MM_0\glue\MM_1$.
\end{remark}
\section{Properties of recollements.}\label{properties}
\setlength{\epigraphwidth}{.75\textwidth}
\epigraph{
``Do what thou wilt'' shall be the whole of the Law.
The study of this Book is forbidden. It is wise to destroy this copy after the first reading.
Whosoever disregards this does so at his own risk and peril.}{\textsc{Ankh-ef-en-Khonsu i}}
\setlength{\epigraphwidth}{\DefaultEpigraphWidth}
In this section we address associativity issues for the $\glue$ operation: it is a somewhat subtle topic, offering examples of several non-trivial constructions even in the classical geometric case: it is our opinion that in a stable setting the discussion can be clarified by simple, well-known categorical properties. 

We start proving a generalization of \cite{Banagl,BBDPervers} where it is stated that the gluing operation can be iterated in a preferential way determined by a \emph{stratification} of an ambient space $X$. This result hides in fact an associativity property for the gluing operation, in a sense which our \abbrv{Thm}\refbf{geomgluing} below makes precise. 

Suitably abstracted to a stable setting, a similar result holds true, once we are given a \emph{Urizen compass} (a certain shape of diagram like in \abbrv{Def} \refbf{defUrizen}, implying certain relations and compatibilities between different recollements, which taken together ensure associativity).
\subsection{Geometric associativity of the gluing.}
An exhaustive account for the theory of stratified spaces can be found in \cite{pflaum2001analytic,Banagl,Wein}. Here, since we do not aim at a comprehensive treatment, we restrict to a sketchy recap of the basic definitions.

A \emph{stratified space} of length $n$ consists of a pair $(X, \textsf{s})$ where 
\begin{equation}\label{stratas}
\textsf{s} \quad : \quad \varnothing = U_{-1} \subset U_0\subset \dots\subset U_n\subset X = U_{n+1}
\end{equation} 
is a chain of closed subspaces of a space $X$, subject to various technical assumptions which ensure that the homology theory we want to attach to $(X, \mathsf{s})$ is ``well-behaved'' in some sense.

All along the following section, we will denote a \emph{pure stratum} of a stratified space $(X, \mathsf{s})$ the set-theoretical difference $E_i = U_i\setminus U_{i-1}$.
\begin{remark}
The definition is intentionally kept somewhat vague in various respects, first of all about the notion of ``space'': the definition of stratification can obviously be given in different contexts (topological spaces, topological manifolds, \textsc{pl}-manifolds,\@\xspace\dots) according to the needs of the specific theory we want to build; when the stratification $\textsf{s}$ is clear from the context, we indulge to harmless, obvious abuses of notation.
\end{remark}
The associativity properties of $\glue$ are deeply linked with the presence of a stratification on a space $X$, in the sense that a stratification $\mathsf{s}$ is what we need to induce additional recollements ``fitting nicely'' in the diagram of inclusions determined by $\mathsf{s}$. These recollements define a unique $t$\hyp{}structure $\tee_0\glue\cdots \glue \tee_n$, given $\tee_i$ on the derived categories of the pure strata.

To motivate the shape and the strength of the abstract conditions ensuring associativity of $\glue$, exposed in \S\refbf{abstractassoc}, and in particular the definition of a Urizen compass \refbf{defUrizen}, we have to dig into deep in the argument sketched in the geometric case in \cite[\textbf{2.1.2-3}]{BBDPervers}: we start by recalling
\begin{theorem}\cite[p. \textbf{158}]{Banagl}\label{geomgluing}
Let $(X, \textsf{s})$ be a stratified space, $\{E_0,\dots, E_n\}$ the set of its pure strata, and $\tee_i$ be a set of $t$\hyp{}structures, one on each $\D(E_i)$, for $i=0,\dots, n$.

Then there exists a uniquely determined $t$\hyp{}structure $\tee_0\glue \cdots \glue \tee_n$ on $\D(X)$, obtained by an iterated gluing operation as the parenthesization 
$(\cdots ((\tee_0\glue \tee_1)\glue \tee_2) \glue  \cdots\glue  \tee_{n-1})\glue \tee_n$. Following Notation \refbf{notat:recdec} we will refer to the pair $(\D(X), \tee_0\glue \cdots \glue \tee_n)$ as $\D(E_0)\glue \cdots \glue \D(E_n)$.
\end{theorem}
\begin{proof}
A stratification of $X$ as in (\refbf{stratas}) induces a certain triangular diagram $\dgrm{G}_n$ of the following form, where all maps $i_k$ are inclusions of the closed subspaces $U_k$ of $\mathsf{s}$, and all $j_k$ are inclusions of the pure strata $E_k$: in the notation above we obtain 
\begin{equation}\label{pyram}
\xymatrix@!R=.25mm@!C=.25mm{
  	& 	& 	& 	& X \\
  	& 	& 	& U_n	& 	&  \\
  	& 	& \iddots	& 	& 	& 	&  \\
  	& U_1	& 	& \iddots	& 	& 	& 	&  \\
  E_0	& 	& E_1	& 	& \boxed{\dgrm{G}_n}	& 	& E_n	& 	& E_{n+1}.
  \ar@{<-} "1,5";"2,4" _{i_n}
  \ar@{<-} "2,4";"3,3" _{i_{n-1}}
  \ar "4,2";"3,3" ^{i_1}
  \ar "5,1";"4,2" ^{i_0}
  \ar "5,3";"4,2" _{j_0}
  \ar "5,7";"2,4" _{j_{n-1}}
  \ar "5,9";"1,5" _{j_n}
}
\end{equation} This diagram can clearly be defined inductively starting from $n=1$ (the diagram of inclusions as in Example \refbf{geomrecoll}).  Given this evident recursive nature, it is sufficient to examine the case $n=2$ of a stratification $U_0\subset U_1\subset X$, depicted as\footnote{Here and for the rest of the section, drawing large diagrams of stable categories, we adopt the following shorthand: every edge $h\colon \cate{E}\to \cate{F}$ is decorated with an adjoint triple $h_L\dashv h\dashv h_R\colon \cate{E}\lrlarrows \cate{F}$.}
\begin{equation}
  \xymatrix@!R=4mm@!C=4mm{
    	& 	& \D(X) \\
    	& \D (U_1)	& 	&  \\
    \D (E_0)	& 	& \D(E_1)	& 	& \D (E_2)
    \ar "1,3";"3,5" ^{q=j_1^*}
    \ar "2,2";"1,3" ^{a=i_{1,*}}
    \ar "2,2";"3,3" _{g=j_0^*}
    \ar "3,1";"2,2" ^{f=i_{0,*}}
  }
\end{equation} to notice that the $t$\hyp{}structure $(\tee_0\glue \tee_1)\glue \tee_2$ obtained by iterated gluing construction is 
\begin{align*}
[(\D(E_0)\glue \D(E_1))\glue \D(E_2)]_{\ge 0} &= \left\{
	  G\in\D(X)\left|\footnotesize 
       	\begin{array}{l}
	  qG \in \D(E_2)_{\ge 0}, \\ 
	  a_L G\in [\D(E_0)\glue \D(E_1)]_{\ge 0} 
	  \end{array}\right.
   \right\}\\
& = \left\{
       G\in\D(X) \left|\footnotesize 
       	\begin{array}{l}
			qG \in \D(E_2)_{\ge 0}, \\ 
			g(a_L G)\in \D(E_1)_{\ge 0}, \\ 
			f_L(a_L G) \in \D (E_0)_{\ge 0}
		\end{array}
		\right.
	\right\}\\
\protect{(\refbf{veryshort})} & = \big\{ G\in\D(X) \mid \{q, g a_L, f_L a_L\}(G)\in \D_{\ge 0} \big\}
\end{align*}
The inductive step simply adds another inclusion (and the obvious maps between derived categories) to these data.
\end{proof}
\begin{remark}
In the previous proof, in the case $n=2$, we could have noticed that two ``hidden'' recollement data, given by the inclusions $$(E_1\hookrightarrow X\setminus U_0, E_2\hookrightarrow X\setminus U_0) \text{ and } (E_0\hookrightarrow X, X\setminus U_0\hookrightarrow X)$$ come into play: the refinement of the inclusions in the diagram above induces an analogous refinement which passes to the derived $\infty$\hyp{}categories,
\begin{equation}\label{puregeom}
    \xymatrix@!R=4mm@!C=4mm{
      	& 	& \D(X) \\
      	& \D (U_1)	& 	& \D (X\setminus U_0) \\
      \D (E_0)	& 	& \D(E_1)	& 	& \D (E_2)
      \ar "1,3";"2,4" ^{u}
      \ar@{} "2,2";"2,4" |{\uno}
      \ar^a "2,2";"1,3" 
      \ar_g "2,2";"3,3" 
      \ar^k "2,4";"3,5" 
      \ar^f "3,1";"2,2" 
      \ar_h "3,3";"2,4" 
    }
\end{equation}of functors between derived $\infty$\hyp{}categories on the pure strata. These data induce two additional recollements, $(k,h)$ and $(u, a\circ f)$ which we can use to define a different parenthesization $\tee_0\glue (\tee_1\glue \tee_2)$.
\begin{remark}
When all the recollements data in (\refbf{puregeom}) are taken into account, we obtain a graph
\begin{equation}
\xymatrix{
	&   &  \D(X)\ar[dr]^u\ar[dl]_{a_L} &&\\
&  \D(U_1)\ar[dr]^g\ar[dl]_{f_L}  &   &  \D(X\setminus U_0)\ar[dr]^k\ar[dl]_{h_L} & \\
 \D(E_0)  &   &  \D(E_1) &   &  \D(E_2)
}
\end{equation}
\index{left\hyp{}winged| see {Geometric recollement}}
\index{Recollement!geometric ---}
called the \emph{left-winged} diagram associated with (\refbf{puregeom}), and defined by taking the left-most adjoint in the string $(-)_L \dashv (-) \dashv (-)_R$, when descending each left ``leaf'' of the tree represented in diagram (\refbf{puregeom}). In a completely similar fashion we can define the \emph{right-winged} diagram of (\refbf{puregeom}). We refer to these diagrams as (\textbf{l-}\refbf{puregeom}) and (\textbf{r-}\refbf{puregeom}) respectively.
\end{remark}
It is now quite natural to speculate about some sort of \emph{comparison} between the two recollements $(\tee_0\glue \tee_1)\glue \tee_2$ and $\tee_0\glue (\tee_1\glue \tee_2)$: in fact we can prove with little effort (once the phenomenon in study has been properly clarified) that the two $t$\hyp{}structures are equal, since the square 
\begin{equation}
	\xymatrix{
	  E_1	& X\setminus U_0 \\
	  U_1	& X
	  \ar "1,1";"1,2" 
	  \ar "1,1";"2,1" 
	  \ar "1,2";"2,2" 
	  \ar "2,1";"2,2" 
	}
\end{equation}is a {\color{black} fiber product} (in a suitable category of spaces) {\color{black} of a proper map with an open embedding}, and so there is a ``change of base'' morphism  $u\circ a \cong h\circ g$ which induces {\color{black}invertible 2-cells $g\circ a_L\cong h_L\circ u$ and $g\circ a_R\cong h_R\circ u$} filling the square \uno in diagram (\refbf{puregeom}): this is a particular instance of the so-called \emph{Beck-Chevalley condition} for a commutative square, which we now adapt to the $\infty$\hyp{}categorical setting.
\end{remark}
\index{Beck\hyp{}Chevalley condition}
\begin{definition}[Beck-Chevalley condition]\label{beckchev}
Consider the square
\begin{equation}\label{bcsquare}
	\xymatrix@!R=1cm@!C=1cm{
	  \cate{A}	& \cate {B} \\
	  \CC	& \D
	  \ar_g "1,1";"2,1" 
	  \ar|a "1,1";"1,2" 
	  \ar@<-6pt> "1,2";"1,1" _{a_L}
	  \ar@<6pt> "1,2";"1,1" ^{a_R}
	  \ar^u "1,2";"2,2" 
	  \ar|h "2,1";"2,2" 
	  \ar@<-6pt> "2,2";"2,1" _{h_L}
	  \ar@<6pt> "2,2";"2,1" ^{h_R}
	}
\end{equation}
in a $(\infty,2)$\hyp{}category, filled by an invertible 2-cell $\theta\colon u\circ a \cong h\circ g$ and such that $a_L\dashv a, h_L \dashv h$; then the square (\ref{bcsquare}) is said to satisfy the \emph{left Beck-Chevalley property}, or that it is a \emph{left Beck-Chevalley} square (\textsc{lbc} for short) if the canonical 2-cell 
\begin{equation}
\hat{\theta} \quad : \quad h_L \circ u \overset{h_L u * \eta}\Longrightarrow h_L \circ u\circ a\circ a_L \overset{h_L * \theta * a_L}\Longrightarrow h_L \circ h\circ g\circ  a_L \overset{\epsilon * g a_L}\Longrightarrow g\circ a_L
\end{equation}
is invertible as well. Similarly, when $a\dashv a_R, h\dashv h_R$ we define the 2-cell
\begin{equation}
\tilde{\theta} \quad : \quad 	g \circ a_R \overset{\eta * g a_R }\Longrightarrow h_R \circ h \circ g \circ a_R \overset{h_R * \theta * a_R}\Longrightarrow h_R \circ u\circ  a\circ  a_R \overset{h_R u * \epsilon}\Longrightarrow h_R \circ u
\end{equation}
and we say that the square (\ref{bcsquare}) is \emph{right Beck-Chevalley} (\textsc{rbc} for short) when $\tilde{\theta}$ is invertible. {\color{black}We will say that the square  (\ref{bcsquare}) is \emph{Beck-Chevalley} (\textsc{bc} for short) when it is both left and right beck-Chevalley.}
\end{definition}
In light of this property enjoyed by diagram \uno in (\refbf{puregeom}) it's rather easy to show that the two left classes
\begin{align*}
\big[ \big(\D(E_0)\glue \D(E_1)\big)\glue \D(E_2)\big]_{\ge 0} &= \big\{ G\in\D(X) \mid \{ku, g a_L, f_L a_L\}(G)\in \D_{\ge 0} \big\}\\
 \big[ \D(E_0)\glue \big(\D(E_1)\glue \D(E_2)\big)\big]_{\ge 0} &= \big\{ G\in\D(X) \mid \{ku, h_L u, f_L a_L\}(G)\in \D_{\ge 0} \big\}
\end{align*}
coincide up to a canonical isomorphism determined by the Beck-Chevalley 2-cell in \uno of diagram (\refbf{puregeom}). 

As a result, both $[ (\D(E_0)\glue \D(E_1))\glue \D(E_2)]_{\ge 0}$ and $[ \D(E_0)\glue (\D(E_1)\glue \D(E_2)) ]_{\ge 0} $ define the torsion class of the same $t$\hyp{}structure $(\D_{\ge 0}^{012}, \D_{< 0}^{012})$ on $\D(X)$. {\color{black}Since 2-cell in \uno of diagram (\refbf{puregeom}) is both left and right Beck-Chevalley, the analogous statement holds for the right classes, too. We can state this fact as follows.}

\label{roruberalles}
\begin{scholium}
{\color{black}An object $G\in \D(X)$ lies in $\D_{\ge 0}^{012}$ if and only if $l_0G\in \D(E_0)_{\ge 0}, l_1 G\in \D(E_1)_{\ge 0}, l_2 G\in \D(E_2)_{\ge 0}$ where $l_i$ is any choice of a functor $\D(X)\to \D(E_i)$ in the left-winged diagram of (\refbf{puregeom}). An object $G\in \D(X)$ lies in $\D_{< 0}^{012}$ if and only if $r_0G\in \D(E_0)_{< 0}, r_1 G\in \D(E_1)_{< 0}, r_2 G\in \D(E_2)_{< 0}$ where $r_i$ is any choice of a functor $\D(X)\to \D(E_i)$ in the right-winged diagram of (\refbf{puregeom}).}
\end{scholium}
It is now rather easy to repeat the same reasoning with arbitrarily long chains of strata: given a stratified space $(X, \textsf{s})$ we can induce the diagram
\begin{equation}\label{stratases}
\scalebox{.85}{ 	\xymatrix@!R=4mm@!C=4mm{
                	  	& 	& 	& 	& X \\
                	  	& 	& 	& U_n	& 	& X\setminus U_0 \\
                	  	& 	& U_{n-1}	& 	& U_n \setminus U_0	& 	& X\setminus U_1 \\
                	  	& \iddots	& 	& \iddots	& 	& \ddots	& 	& \ddots \\
                	  U_0	& 	& U_1 \setminus U_0	& 	& 	& 	& U_n \setminus U_{n-1}	& 	& X\setminus U_n
                	  \ar@{<-} "1,5";"2,6" 
                	  \ar@{<-} "1,5";"2,4" 
                	  \ar@{<-} "2,4";"3,5" 
                	  \ar@{<-} "2,4";"3,3" 
                	  \ar@{<-} "2,6";"3,7" 
                	  \ar@{<-} "2,6";"3,5" 
                	  \ar@{<-} "3,5";"4,6" 
                	  \ar@{<-} "3,5";"4,4" 
                	  \ar "4,2";"3,3" 
                	  \ar "4,8";"3,7" 
                	  \ar "5,1";"4,2" 
                	  \ar "5,3";"4,4" 
                	  \ar "5,7";"4,6" 
                	  \ar "5,9";"4,8" 
                	}}
\end{equation}where leaves correspond to pure strata of the stratification of $X$, and every square is a pullback of a proper map along an open embedding, 
so that the Beck-Chevalley condition is automatically satisfied {\color{black}by each square in the corresponding diagram $\D(\refbf{stratases})$ of $\infty$-categories of sheaves of the various nodes.}
{\color{black}The diagram $\D(\refbf{stratases})$ 
is equipped with 
recollement data between its adjacent nodes; we can again define the left-winged and right-winged version of $\D(\refbf{stratases})$, which we will refer as $\textbf{l-}\D(\refbf{stratases})$ and $\textbf{r-}\D(\refbf{stratases})$.}

Grouping all these considerations we obtain that
\begin{enumerate}
\item There exist ``compatible'' recollements to give associativity of all the parenthesizations
\begin{equation}
	(\tee_0\glue \cdots\glue \tee_n)_\mathfrak{P}=(\tee_0\glue \cdots\glue \tee_n)_\mathfrak{Q}
\end{equation}
for each $\mathfrak P, \mathfrak Q$ in the set of all possible parenthesizations of $n$ symbols. This is precisely the sense in which, as hinted above, geometric stratifications and recollement data ``interact nicely'' to give canonical isomorphisms between $(\tee_0\glue \cdots\glue \tee_n)_\mathfrak{P}$ and $(\tee_0\glue \cdots\glue \tee_n)_\mathfrak{Q}$, \ie a canonical choice for associativity constraints on the $\glue$ operation.
\item The following characterization for the class $\big( \D(E_0)\glue \cdots \glue \D(E_n)\big)_{\ge 0}$ holds:
\begin{equation}
	\big( \D(E_0)\glue \cdots \glue \D(E_n)\big)_{\ge 0} =
	\Big\{
	  G \mid l_i(G)\in \D(E_i)_{\ge 0}, \hspace{1mm} \forall i=0,\dots,n
	\Big\}
\end{equation}
where $l_i$ is any choice of a functor $\D(X)\to \D(E_i)$ in the left-winged diagram $\textbf{l-}\D(\refbf{stratases})$.

Similarly, the right class $\big( \D(E_0)\glue \cdots \glue \D(E_n)\big)_{< 0} $ can be characterized as the class of objects $G$ such that $r_i(G)\in \D(E_i)_{< 0}$, where $r_i$ is any choice of a functor $\D(X)\to \D(E_i)$ in the right-winged diagram $\textbf{r-}\D(\refbf{stratases})$.
\end{enumerate}
\subsection{Abstract associativity of the gluing.}\label{abstractassoc}
\index{.glue@$\glue$!Associativity of ---}
The geometric case studied above gives us enough information to make an ansatz for a general definition, telling us what we have to generalize, and in which way. 

In an abstract, stable setting we have the following definition, which also generalizes, in some sense, \refbf{def:recol}.

Let $n\ge 2$ be an integer, and let us denote as $\llbracket i,j\rrbracket$ the \emph{interval} between $i,j\in [n]$, \ie, set $\{k\mid i\le k\le j\}\subset [n]=\{0,1,\dots,n\}$ (we implicitly assume $i \le j$ and we denote $\llbracket i,i\rrbracket = \{i\}$ simply as $i$).
\index{Urizen compass}
\begin{definition}[Urizen compass\protect{\footnote{In the complicated cosmogony of W\@. Blake, \emph{Urizen} represents conventional reason and law; it is often represented bearing the same compass of the Great Architect of the Universe postulated by speculative Freemasonry; see for example the painting \emph{The Ancient of Days}, appearing on the frontispiece of the prophetic book ``Europe a Prophecy''.}}]\label{defUrizen}

A \emph{Urizen compass} of length $n$ is an arrangement of stable $\infty$\hyp{}categories, labeled by intervals $I\subseteq [n]$,  and functors in a diagram $\dgrm{G}_n$ of the form
\begin{equation}\label{ponzi}
\xymatrix@!R=3mm@!C=3mm{
  	& 	& 	& 	& \D^{\llbracket 0, n\rrbracket} \\
  	& 	& 	& \iddots	& 	& \ddots \\
  	& 	& \D^{\llbracket0,2\rrbracket}	& 	& \boxed{\dgrm{G}_n}	& 	& \D^{\llbracket n-2,n\rrbracket} \\
  	& \D^{\llbracket 0,1\rrbracket}	& 	& \D^{\llbracket 1,2\rrbracket}	& 	& \iddots	& 	& \D^{\llbracket n-1,n\rrbracket} \\
  \D^0	& 	& \D^1	& 	& \D^2	& 	& \dots	& 	& \D^n
  \ar "1,5";"2,6" 
  \ar "2,4";"1,5" 
  \ar "2,6";"3,7" 
  \ar "3,3";"2,4" 
  \ar "3,3";"4,4" 
  \ar "3,7";"4,8" 
  \ar "4,2";"5,3" 
  \ar "4,2";"3,3" 
  \ar "4,4";"5,5" 
  \ar "4,6";"3,7" 
  \ar "4,8";"5,9" 
  \ar "5,1";"4,2" 
  \ar "5,3";"4,4" 
  \ar "5,5";"4,6" 
  \ar "5,7";"4,8" 
}
\end{equation}
such that the following conditions hold:
\begin{itemize}
\item All the triples $\{\D^{I}, \D^{I\cupdot J}, \D^{J}\}$, where $I,J$ are contiguous intervals,\footnote{Two intervals $I,J\subseteq [n]$ are called \emph{contiguous} if they are disjoint and their union $I\cup J$ is again an interval denoted $I\cupdot J$.} form different recollements $\D^{I}\lrlarrows \D^{I\cupdot J}\lrlarrows \D^{J}$.
\item Every square 
\begin{equation}
	\xymatrix{
	  \D^{\llbracket i,j\rrbracket}	& \D^{\llbracket i,j+1\rrbracket} \\
	  \D^{\llbracket i+1,j\rrbracket}	& \D^{\llbracket i+1,j+1\rrbracket}
	  \ar "1,1";"1,2" 
	  \ar "1,1";"2,1" 
	  \ar "1,2";"2,2" 
	  \ar "2,1";"2,2" 
	}
\end{equation}
{\color{black}is \textsc{bc} in the sense of Definition \refbf{beckchev}.}
\end{itemize}
\end{definition}
Note that each row, starting from the base of the diagram, displays all possible intervals of length $k$. We can think of a Urizen compass as a special kind of directed graph (more precisely, a special kind of rooted oriented tree --a \emph{multitree} if we stipulate that each edge shortens a triple of adjunctions); the root of the tree is the category $\D^{\llbracket 0,\dots, n\rrbracket}$; the leaves are the categories $\{\D^0,\dots, \D^n\}$ (the ``generalized pure strata'').

\begin{theorem}[The northern emisphere theorem\protect{\footnote{In the languages spoken in the northern hemisphere of Tl\"on, ``la célula primordial no es el verbo, sino el adjetivo monosilábico. El sustantivo se forma por acumulación de adjetivos. No se dice luna: se dice \emph{aéreo-claro sobre oscuro-redondo} o \emph{anaranjado-tenue-del cielo} o cualquier otra agregación. \omissis{} Hay objetos compuestos de dos términos, uno de carácter visual y otro auditivo: el color del naciente y el remoto grito de un pájaro. Los hay de muchos: el sol y el agua contra el pecho del nadador, el vago rosa trémulo que se ve con los ojos cerrados, la sensación de quien se deja llevar por un río y también por el sueño. Esos objetos de segundo grado pueden combinarse con otros; el proceso, mediante ciertas abreviaturas, es prácticamente infinito. Hay poemas famosos compuestos de una sola enorme palabra.'' (\cite{Borges1963})}}]
 \label{northern}
A Urizen compass of length $n$ induce canonical isomorphisms between the various parenthesizations of $\tee_0\glue \cdots\glue \tee_n$, giving associativity of the glue operation between $t$\hyp{}structures.
 \end{theorem}
Rephrasing the above result in a more operative perspective, whenever we have a $n$-tuple $\{(\D^i,\tee_i)\}_{i=0,\dots, n}$ of stable $\infty$\hyp{}categories with $t$\hyp{}structure, such that $\{\D^0,\dots, \D^n\}$ are the leaves of a Urizen compass of length $n$, then the gluing operation between $t$\hyp{}structures gives a unique (up to canonical isomorphism) ``glued'' $t$\hyp{}structure on the root $\D^{\llbracket 0,n\rrbracket}$ of the scheme, resulting as
\begin{align}
	\big( \D^0\glue \cdots \glue \D^n\big)_{\ge 0} &=
	\Big\{
	  X\in \D^{\llbracket 0,n\rrbracket} \mid l_i(X)\in \D^i_{\ge 0}, \hspace{1mm} \forall i=0,\dots,n
	\Big\}\notag \\ 
	\big( \D^0\glue \cdots \glue \D^n\big)_{< 0} &=
	\Big\{
	  X\in \D^{\llbracket 0,n\rrbracket}  \mid r_i(X)\in \D^i_{< 0}, \hspace{1mm} \forall i=0,\dots,n
	\Big\}
\end{align}
where $l_i$ is any choice of a path from the root $\D^{\llbracket 0,n\rrbracket}$ to the $i^\text{th}$ leaf in the left-winged diagram of $\mathsf{G}_n$, and $r_i$ is any choice of a path from the root $\D^{\llbracket 0,n\rrbracket}$ to the $i^\text{th}$ leaf in the right-winged diagram of $\mathsf{G}_n$.
\subsection{Gluing $J$-families.}
Our theory of slicings \cite{heart} shows that the set $ \ts (\D)$ of $t$\hyp{}structures on a stable $\infty$\hyp{}category $\D$ carries a natural action of the ordered group of integers. This entails that the most natural notion of a ``family'' of $t$\hyp{}structures is a \emph{equivariant} $J$-family of $t$\hyp{}structures, namely an equivariant map $J\to  \ts (\D)$ from another $\mathbb{Z}$-poset $J$.

The formalism of equivariant families allows to unify several constructions in the classical theory of $t$\hyp{}structures: in particular
\begin{quote}
The \emph{semiorthogonal decompositions} of \cite{BO, Kuz} are described as precisely those $J$-families $\tee\colon J\to  \ts (\D)$ taking values on fixed points of the $\mathbb{Z}$-action; these are equivalently characterized as
\begin{itemize}
\item the \emph{stable} $t$\hyp{}structures, where the torsion and torsionfree classes are themselves stable $\infty$\hyp{}categories;
\item the equivariant $J$-families where $J$ has the trivial action.
\end{itemize}
\end{quote}
And again
\begin{quote}
The datum of a single $t$\hyp{}structure $\tee\colon \{*\}\to  \ts (\D)$ is equivalent to the datum of a whole \emph{$\mathbb{Z}$-orbit} of $t$\hyp{}structures, namely an equivariant map $\mathbb{Z}\to  \ts (\D)$.
\end{quote}
In light of these remarks, given a recollement $(i,q)\colon \D^0 \lrlarrows  \D \lrlarrows  \D^1$ it is natural to define the gluing of two $J$-families 
\begin{equation}
	\xymatrix{
	   \ts (\D^0)	& J	&  \ts (\D^1)
	  \ar_(.3){\tee_0} "1,2";"1,1" 
	  \ar^(.3){\tee_1} "1,2";"1,3"
	}
\end{equation}
to be the $J$-family $\tee_0\glue\tee_1\colon J\to  \ts (\D)\colon j\mapsto \tee_0(j)\glue\tee_1(j)$.

It is now quite natural to ask how does the gluing operation interact with the two situations above: is the gluing of two $J$-families again a $J$-family? As we are going to show, the answer to this question is: yes. Indeed, it's easy to see that the gluing operation is an equivariant map, by recalling that $(\EE_0\glue \EE_1)[1] = \{f\in\hom(\D)\mid f[-1]\in \EE_0\glue \EE_1\}$, and that all of the functors $q,i_L,i_R$ preserves the pullouts (and so commute with the shift). We have
\begin{align*}
(\EE_0\glue \EE_1)[1] &= \{ f\in\hom(\D)\mid \{q, i_L\}(f[-1])\in \EE\}\\
&= \{ f\in\hom(\D)\mid q(f[-1])\in \EE_1,\hspace{1mm} i_L(f[-1])\in\EE_0\}\\
&= \{ f\in\hom(\D)\mid q(f)[-1]\in \EE_1,\hspace{1mm} i_L(f)[-1]\in\EE_0\}\\
&= \{ f\in\hom(\D)\mid q(f) \in \EE_1[1],\hspace{1mm} i_L(f)\in \EE_0[1]\}\\
&= \EE_0 [1]\glue \EE_1[1].
\end{align*}

Given this, it is obvious that given two semiorthogonal decompositions $\tee_i\colon J\to  \ts (\D_i)$ on $\D^0, \D^1$, the $J$-family $\tee_0\glue\tee_1$ is again a semiorthogonal decomposition on $\D$ (the trivial action on $J$ remains the same; it is also possible to prove directly that if $\EE_0,\EE_1$ are left parts of two exact normal torsion theories $\fF_0,\fF_1$ on $\D^0, \D^1$, then the gluing $\EE_0\glue \EE_1$ is the left part of the exact normal torsion theory $\fF_0\glue\fF_1$ on $\D$).
In some sense at the other side is the gluing of two $\mathbb{Z}$-orbits $\tee_0,\tee_1\colon \mathbb{Z}\to  \ts (\CC)$ on $\D^0$ and $\D^1$. Namely, the glued $t$\hyp{}structure $\tee_0\glue\tee_1$ on $\D$ is the $\mathbb{Z}$-orbit $(\tee_0 \glue \tee_1)[k] = \tee_0[k]\glue \tee_1[k]$. 

\medskip 
The important point here is that this construction can be framed in the more general context of \emph{perversity data} associated to a recollement, which we now discuss in the attempt to generalize at least part of the classical theory of ``perverse sheaves'' to the abstract, $\infty$\hyp{}categorical and torsio-centric setting.
\begin{definition}[Perversity datum]\label{def.perversity datum}
Let $p\colon \{0,1\}\to\mathbb Z$ be any function, called a \emph{perversity datum}; suppose that a recollement
$$(i,q)\colon \D^0 \lrlarrows  \D  \lrlarrows \D^1$$
is given, and that $\tee_0,\tee_1$ are $t$\hyp{}structures on $\D^0, \D^1$ respectively. We define the ($p$-)\emph{perverted $t$\hyp{}structures} on $\D^0,\D^1$ as
\begin{gather*}
\prescript{p}{}{\tee}_0 = \tee_0[p(0)] = (\D^0_{\ge p(0)}, \D^0_{<p(0)})\\
\prescript{p}{}{\tee}_1 = \tee_1[p(1)] = (\D^1_{\ge p(1)}, \D^1_{<p(1)})
\end{gather*}
\end{definition}
\begin{definition}[Perverse objects]
Let $p$ be a perversity datum, in the notation above; the \emph{($p$-)glued $t$\hyp{}structure} is the $t$\hyp{}structure $\prescript{p}{}{(\tee_0\glue\tee_1)} = \prescript{p}{}{\tee}_0\glue \prescript{p}{}{\tee}_1$. The heart of the $p$-perverted $t$\hyp{}structure on $\D$ is called the ($\infty$-)category of \emph{($p$-)perverse objects} of $\D$.
\end{definition}
Notice that saying  ``the category of $p$-perverse objects of $\D$'' is an abuse of notation: this category indeed does not depend only on $\D$ and $p$, but on all of the recollement data and on the $t$\hyp{}structures $\tee_0$ and $\tee_1$. Also notice how for a constant perversity datum $p(0)=p(1)=k$, the $p$-perverted $t$\hyp{}structure is nothing but the $t$\hyp{}structure $\tee_0\glue\tee_1$ shifted by $k$.

\medskip 

We can extend the former discussion to the gluing of a whole $n$-tuple of $t$\hyp{}structures, using a Urizen compass:
\begin{remark}\label{rem.urizen}
In the case of a Urizen compass of dimension $n$ (diagram \refbf{ponzi}), whose leaves are the categories $\{\D^0 ,\dots, \D^n\}$, each endowed with a $t$\hyp{}structure $\tee_i$; a perversity function $p\colon \{0,\dots,n\}\to \mathbb{Z}$ defines a perverted $t$\hyp{}structure
\begin{equation}
\prescript{p}{}{(\tee_0\glue \cdots \glue \tee_n)} = \tee_0[p(0)]\glue \tee_1[p(1)]\glue \cdots \glue \tee_n[p(n)]
\end{equation}
which is well-defined in any parenthesization thanks to the structure defining the Urizen compass.
This result  immediately generalizes to the case of a Urizen compass of $J$-families of $t$\hyp{}structures, $\tee_i\colon J\to  \ts (\D_i)$, with $i=0,\dots, n$. Indeed perversity data act on $J$-equivariant families of $t$\hyp{}structures by
\begin{equation}
\prescript{p}{}{\tee}_i(j) = \tee_i(j)[p(i)] = (\D^i_{\ge j+p(i)}, \D^i_{<j+p(i)}).
\end{equation}
This way, a $J$-perversity datum $p\colon \{0,\dots,n\}\to \mathbb{Z}$ induces a $p$-perverted $t$\hyp{}structure
\begin{equation}
\prescript{p}{}{(\tee_0\glue \cdots \glue \tee_n)}=\prescript{p}{}{\tee}_0\glue \cdots \glue \prescript{p}{}{\tee}_n\,\colon\, J\to \ts(\D^{\llbracket 0,n\rrbracket})
\end{equation}
on $\D^{\llbracket 0,n\rrbracket}$.
\end{remark}

\begin{remark}[Gluing of slicings.]\index{Slicing!gluing of ---s}\index{Slicing}
Recall that a \emph{slicing} on a stable $\infty$\hyp{}category $\D$ consists on a $\mathbb{R}$-family of $t$\hyp{}structures $\tee\colon \mathbb{R}\to  \ts (\D)$, where $\mathbb R$ is endowed with the usual total order. This means that we are given $t$\hyp{}structures $\tee_\lambda = (\D_{\ge \lambda}, \D_{ < \lambda})$, one for each $\lambda\in\mathbb R$, such that $\tee_{\lambda +1} = \tee_\lambda[1]$. Slicings on $\D$ are part of the abstract definition of a $t$-\emph{stability} on a triangulated (or stable) category $\D$, see \cite{Brid,GKR}.

Grouping together all the above remarks, we obtain that the gluing of two slicings $\tee_i\colon \mathbb{R}\to  \ts (\D^i)$ gives a slicing on $\D$ every time $\D^0\lrlarrows \D\lrlarrows \D^1$ is a recollement on $\D$. Moreover, if $p\colon \{0,1\}\to \mathbb{Z}$ is a perversity datum, we have a corresponding notion of \emph{$p$-perverted slicing} on $\D$. More generally one has a notion of $p$-perverted slicing on $\D^{\llbracket 0,n\rrbracket}$ induced by a pervesity datum $p$ and by and by a Urizen compass of slicings $\dgrm{G}_n$.
\end{remark}
\paragraph{Acknowledgements}
Version 1 of the present paper is sensibly different from the present one; the unexpected (and actually undue) symmetric behavior of stable recollements (Lemma \textbf{4.3} of version 1, therein called the \emph{Rorschach lemma}\footnote{Walter Joseph Kovacs, also known as Rorschach (New York 1940 – Antarctica 1985).})  turned out to be the far reaching consequence of a typo in one of the commutative diagrams on page 9. This has now been corrected (\ie, Lemma \textbf{4.3}, together with all its corollaries, has been removed). 

Luckily, this was only minimally affecting the remaining part of the article, which has now been revised accordingly. In particular the section on the associative properties of recollements has been expanded, some additional examples have been added, and several other minor typos have been corrected.

\newcommand{\arXivPreprint}[1]{arXiv preprint \href{http://arxiv.org/abs/#1}{arXiv:#1}}

\newcommand{\etalchar}[1]{$^{#1}$}
\providecommand{\bysame}{\leavevmode\hbox to3em{\hrulefill}\thinspace}
\providecommand{\MR}{\relax\ifhmode\unskip\space\fi MR }
\providecommand{\MRhref}[2]{%
  \href{http://www.ams.org/mathscinet-getitem?mr=#1}{#2}
}
\providecommand{\href}[2]{#2}


\begin{thebibliography}{AHKL11}

\bibitem[AHKL11]{hugel2011recollements}
L.~Angeleri~H{\"u}gel, S.~Koenig, and Q.~Liu, \emph{Recollements and tilting
  objects}, {J}ournal of {P}ure and {A}pplied {A}lgebra \textbf{215} (2011),
  no.~4, 420--438.

\bibitem[Ban07]{Banagl}
M.~Banagl, \emph{Topological invariants of stratified spaces}, Springer
  Monographs in Mathematics, Springer, Berlin, 2007. \MR{2286904 (2007j:55007)}

\bibitem[BBD82]{BBDPervers}
A.~A. Beilinson, J.~Bernstein, and P.~Deligne, \emph{Faisceaux pervers},
  Analysis and topology on singular spaces, {I} ({L}uminy, 1981), Ast\'erisque,
  vol. 100, Soc. Math. France, Paris, 1982, pp.~5--171. \MR{751966 (86g:32015)}

\bibitem[BO95]{BO}
A.~I. {Bondal} and D.~O. {Orlov}, \emph{Semiorthogonal decomposition for
  algebraic varieties}, ArXiv e-prints
  \href{http://arxiv.org/abs/alg-geom/9506012v1}{alg-geom/9506012} (1995), 55
  pp.

\bibitem[Bor44]{Borges1963}
J.~L. Borges, \emph{Ficciones}, Editorial Sur, Buenos Aires, 1944.

\bibitem[BP13]{bazzoni2013recollements}
S.~Bazzoni and A.~Pavarin, \emph{Recollements from partial tilting complexes},
  Journal of Algebra \textbf{388} (2013), 338--363.

\bibitem[BR07]{Beligiannisreiten}
A.~Beligiannis and I.~Reiten, \emph{Homological and homotopical aspects of
  torsion theories}, Mem. Amer. Math. Soc. \textbf{188} (2007), no.~883,
  viii+207. \MR{2327478 (2009e:18026)}

\bibitem[Bri07]{Brid}
T.~Bridgeland, \emph{Stability conditions on triangulated categories}, Ann. of
  Math. (2) \textbf{166} (2007), no.~2, 317--345. \MR{2373143 (2009c:14026)}

\bibitem[C{\etalchar{+}}14]{chen2014recollements}
H.~Chen et~al., \emph{Recollements of derived categories \textsc{iii}:
  Finitistic dimensions}, \arXivPreprint{1405.5090} (2014), 26.

\bibitem[ER77]{BHS}
K.~Elliger and W.~Rudolph (eds.), \emph{{B}iblia {H}ebraica {S}tuttgartensia},
  editio quinta emendata ed., Deutsche Bibelgesellschaft, Stuttgart, 1977.

\bibitem[FL15a]{heart}
D.~{F}iorenza and F.~{L}oregi\`an, \emph{Hearts and towers in stable
  $\infty$-categories}, \arXivPreprint{1501.04658} (2015), 24.

\bibitem[FL15b]{FL0}
\bysame, \emph{$t$-structures are normal torsion theories}, Applied Categorical
  Structures (2015), 1--28.

\bibitem[GKR04]{GKR}
A.~L. Gorodentsev, Sergej~A. Kuleshov, and Aleksei~N. Rudakov,
  \emph{{$t$}-stabilities and {$t$}-structures on triangulated categories},
  Izv. R.. Akad. Nauk Ser. Mat. \textbf{68} (2004), 117--150. \MR{2084563
  (2005j:18008)}

\bibitem[GM80]{goresky1980intersection}
M.~Goresky and R.~MacPherson, \emph{Intersection homology theory}, Topology
  \textbf{19} (1980), no.~2, 135--162.

\bibitem[GM83]{goresky1983intersection}
\bysame, \emph{Intersection homology \textsc{ii}}, Inventiones Mathematicae
  \textbf{72} (1983), no.~1, 77--129.

\bibitem[Han14]{han2014recollements}
Y.~Han, \emph{Recollements and hochschild theory}, Journal of Algebra
  \textbf{397} (2014), 535--547.

\bibitem[Hew95]{hewitt1995georgian}
B.G. Hewitt, \emph{Georgian: A structural reference grammar}, London Oriental
  and African language library, John Benjamins Publishing Company, 1995.

\bibitem[HJ10]{Hol}
T.~Holm and P.~J{\o}rgensen, \emph{Triangulated categories: definitions,
  properties, and examples}, Triangulated categories, London Math. Soc. Lecture
  Note Ser., vol. 375, Cambridge Univ. Press, Cambridge, 2010, pp.~1--51.
  \MR{2681706 (2012i:18011)}

\bibitem[HQ14]{nrecol1}
Y.~Han and Y.~Qin, \emph{Reducing homological conjectures by $n$-recollements},
  \arXivPreprint{1410.3223} (2014), 22.

\bibitem[Jan65]{jans1965}
J.~P. Jans, \emph{Some aspects of torsion.}, Pacific J. Math. \textbf{15}
  (1965), no.~4, 1249--1259.

\bibitem[KS90]{KS1}
M.~Kashiwara and P.~Schapira, \emph{Sheaves on manifolds}, Grundlehren der
  Mathematischen Wissenschaften, no. 292, Springer-Verlag - Grundlehren der
  mathematischen Wissenschaften, Berlin, 1990. \MR{1299726 (95g:58222)}

\bibitem[Kuz11]{Kuz}
A.~Kuznetsov, \emph{Base change for semiorthogonal decompositions}, Compos.
  Math. \textbf{147} (2011), no.~3, 852--876. \MR{2801403}

\bibitem[KW01]{kiehl2001weil}
R.~Kiehl and R.~Weissauer, \emph{Weil conjectures, perverse sheaves and
  {$\ell$}-adic {F}ourier transform}, Ergebnisse der Mathematik und ihrer
  Grenzgebiete. 3. Folge. A Series of Modern Surveys in Mathematics., vol.
  \textbf{42}, Springer-Verlag, Berlin, 2001. \MR{1855066 (2002k:14026)}

\bibitem[Lur11]{LurieHA}
J.~Lurie, \emph{Higher algebra}, online version May 18, 2011.

\bibitem[Nic08]{ttftriples}
Pedro Nicolas, \emph{On torsion torsionfree triples}, 2008, p.~184.

\bibitem[Pfl01]{pflaum2001analytic}
M.~J. Pflaum, \emph{Analytic and geometric study of stratified spaces}, Lecture
  Notes in Mathematics 1768, vol. 1768, Springer Berlin Heidelberg, 2001.

\bibitem[PS88]{parshall1988derived}
B.~Parshall and L.~Scott, \emph{Derived categories, quasi-hereditary algebras,
  and algebraic groups}, Carlton University Mathematical notes \textbf{3}
  (1988), 1--104.

\bibitem[Qin15]{nrecol2}
Y.~Qin, \emph{{J}ordan-{H}\"older theorems for derived categories of derived
  discrete algebras}, \arXivPreprint{1506.08266} (2015), 18.

\bibitem[RS76]{roth1976deutsches}
F.~Roth-Scholtz, \emph{Deutsches {T}heatrum {C}hemicum}, vol.~2, G. Olms, 1976.

\bibitem[Sch]{SchwedeGlobal}
S.~Schwede, \emph{Global homotopy theory}.

\bibitem[Wei94]{Wein}
S.~Weinberger, \emph{The topological classification of stratified spaces},
  Chicago Lectures in Mathematics, University of Chicago Press, Chicago, IL,
  1994. \MR{1308714 (96b:57024)}

\end{thebibliography}
\end{document}